\documentclass[11pt]{amsart}
\usepackage{etex}
\usepackage{verbatim} 
\usepackage{arydshln} 
\usepackage{rotating}
\usepackage{amsmath}
\usepackage{amssymb}
\usepackage{amscd}
\usepackage{amsfonts}
\usepackage{enumerate}
\usepackage{color}
\usepackage{bm}
\usepackage{easybmat} 
\usepackage{graphicx}
\usepackage{subfigure}
\usepackage{fullpage}
\usepackage[makeroom]{cancel}
\usepackage{tikz}
\usepackage{mathrsfs}
\usepackage{tikz-qtree}
\usetikzlibrary{trees}

\newtheorem{Thm}{Theorem}[section]
\newtheorem{Cor}[Thm]{Corollary}
\newtheorem{Lemma}[Thm]{Lemma}
\newtheorem{Prop}[Thm]{Proposition}

\theoremstyle{definition}

\newtheorem{Rmk}[Thm]{Remark}
\newtheorem{Example}[Thm]{Example}

\newtheorem*{notation}{Notation}

\numberwithin{equation}{section}

\newcommand{\overbar}[1]{\mkern 1.5mu\overline{\mkern-1.5mu#1\mkern-1.5mu}\mkern 1.5mu}

\def\frakP{\mathfrak{P}}

\def\frakT{\mathfrak{T}}

\def\bbc{\mathbb{C}}

\def\bbn{\mathbb{N}}

\def\bbz{\mathbb{Z}}

\def\calH{\mathcal{H}}
\def\calI{\mathcal{I}}

\def\calO{\mathcal{O}}
\def\calP{\mathcal{P}}

\def\calS{\mathcal{S}}
\def\calU{\mathcal{U}}

\def\calZ{\mathcal{Z}}
\def\calS{\mathcal{S}}
\def\calT{\mathcal{T}}
\def\calQ{\mathcal{Q}}

\def\lra{\longrightarrow}

\def\x{\times}

\def\aut{\mathrm{Aut}}

\def\id{\mathrm{id}}

\def\sp{\mathrm{sp}}

\def\E{\mathrm{(E)}}

\def\GR{\mathrm{GR}}

\def\Sym{\mathsf{Sym}}
\def\FSym{\mathsf{FSym}}
\def\FAlt{\mathsf{FAlt}}

\def\supp{\mathrm{supp}}
\def\Es{\mathrm{supp_e}}

\def\supp{\mathrm{supp}\:}
\def\Es{\mathrm{Esupp}\:}
\def\pt{\mathsf{p}}
\def\qt{\mathsf{q}}
\def\tp{\tilde{\phi}}
\def\ES{\mathrm{ES(\phi)}}
\def\E{\mathrm{E}}

\def\boxit#1{\vbox{\hrule\hbox{\vrule\kern3pt
     \vbox{\kern3pt#1\kern3pt}\kern3pt\vrule}\hrule}}
\def\sqbox#1{
\gdef\boxcont{\rm #1}
  \setbox4=\vbox{\hsize 38pc \noindent \strut \boxcont\strut}
  \par\centerline{\boxit{\box4}}\par}

\def\colr#1{\textcolor{red}{#1}}
\def\colb#1{\textcolor{blue}{#1}}

\begin{document}
\title[Growth rate of endomorphisms of Houghton's groups $\calH_n$]
{Growth rate of endomorphisms of Houghton's groups}

\author{Jong Bum Lee}
\address{Department of Mathematics, Sogang University, Seoul 121-742, KOREA}
\email{jlee@sogang.ac.kr}

\author{Sang Rae Lee}
\address
{Department of Mathematics, Texas A\&M University, College Station, Texas 77843, USA}
\email{srlee@math.tamu.edu}

\subjclass[2000]{}%
\keywords{Houghton's groups, Growth rate of endomorphisms}

\begin{abstract}
A Houghton's group $\calH_n$ consists of translations at infinity of a $n$ rays of discrete points on the plane. In this paper we study the growth rate of endomorphisms of Houghton's groups. We show that if the kernel of an endomorphism $\phi$ is not trivial then the growth rate $\GR(\phi)$ equals either $1$ or the spectral radius of the induced map on the abelianization. It turns out that every monomorphism $\phi$ of $\calH_n$ determines a unique natural number $\ell$ such that $\phi(\calH_n)$ is generated by translations with the same translation length $\ell$. We use this to show that $\GR(\phi)$ of a monomorphism $\phi$ of $\calH_n$ is precisely
$\ell$ for all $2\leq n$.
\end{abstract}
\date{\today}
\maketitle



\section{Introduction}\label{sec:introd}

Let $G$ be a finitely generated group with a generating set $A$.
Let $\phi:G\to G$ be an endomorphism.
For any $g\in G$,
let $|g|$ denote the \emph{word length} of $g$, that is, the minimum length of a word over $A\cup A^{-1}$
which represents $g$.
Then the {\bf growth rate} of $\phi$ is defined to be (\cite{B78})
$$
\GR(\phi):=\sup\left\{\limsup_{k\to\infty}(|\phi^k(g)|)^{1/k}: g\in G\right\}.
$$
The growth rate is well-defined,
i.e., independent of the choice of a set of generators  (\cite[p.~\!114]{KH}).

The problem of determining the growth rate of a group endomorphism was initiated by
R. Bowen in (\cite{B78}).  The growth rate of an endomorphism is related to algebraic entropy and topological entropy (\cite{B78}, \cite{KH}). Algebraic entropy of $\phi$ is defined by $h_{\mathrm{alg}}(\phi)= \log \GR(\phi)$. Note that $h_{\mathrm{alg}}(\phi)$ provides a lower bound for the topological entropy of a continuous self map $f$ on a compact connected manifold $M$ which induces the endomorphism $\phi$ of $\pi_1(M)$.

A group theoretic approach is discussed in \cite{FFK} including the result that $\GR(\phi)$ is finite and bounded by the maximum length of the image of a generator. In case $\phi$ is an automorphism of a nilpotent group it is shown in \cite{K} that $\GR(\phi)$ coincides with the growth rate of the induced automorphism on its abelianization. In \cite{FJL}, the first author extends this to all endomorphisms of nilpotent groups. In the same article it was proven that $\GR(\phi)$ is an algebraic number if $\phi$ is an endomorphism of a torsion-free nilpotent or lattices of Sol.

For $n\in \bbn$, a Houghton's group $\calH_n$ consists of \emph{eventual translations} on a disjoint union of $n$ copies of $\bbn$, each arranged along a ray emanating from the origin in the plane. The group $\calH_2$ is finitely generated but not finitely presented. In general, by the work of K. Brown (\cite{B}), $\calH_n$ has finiteness type $F_{n-1}$ but not $F_n$. For each $n$, $\calH_n$ fits into the short exact sequence.
$$
1\lra \FSym_n\lra\calH_n \lra\bbz^{n-1}\lra1.
$$
where $\FSym_n$ consists of permutations on the underlying set with finite supports. Note that every $\calH_n$ contains all finite groups. Our main result is the following.

\begin{Thm}
Suppose $\phi$ is an endomorphism of $\calH_n$, $2\leq n$. Every monomorphism $\phi$ of $\calH_n$ determines $\ell \in \bbn$ such that $\GR(\phi)=\ell$. If $\phi$ is not a monomorphism then $\GR(\phi)$ equals either $1$ or the spectral radius of the induced map on the abelianization.
\end{Thm}

In Section \ref{sec:Houghton} we define Houghton's groups and review basic facts including explicit presentations. In Section \ref{sec:auto} we deal with automorphisms of $\calH_n$. We calculate $\GR(\phi)$ by using the structure of $\aut(\calH_n)$ which  is well understood. If $\phi$ is an endomorphism with non trivial kernel we can reduce the calculation of $\GR(\phi)$ to the problem of growth rate of endomorphisms on finitely generated free abelian groups. In Section \ref{sec:not monic} we use the result of \cite{FJL} to prove the second part of our main theorem above. The rest of paper is used to understand monomorphisms of $\calH_n$. For $3\leq n$, $\calH_n$ is generated by $g_2, \cdots, g_{n}$ where $g_i$ translates points on the first ray toward $i^{th}$ ray by $1$. We use the relations of $\calH_n$ together with the ray structure of the underlying set to characterize the behavior of monomorphisms in several steps carefully. It turns out that each monomorphism $\phi$ determines $\ell\in \bbn$ such that $\phi$ maps each generator $g_i$ of $\calH_n$ to a translation of length $\ell$. This character of a monomorphism is essential for us to understand iterations $\phi^k$ applied to an element $f$ with a finite support, which is the main obstruction in calculating $\GR(\phi)$. Perhaps Houhgton's groups are first examples of groups where relations were used extensively to calculate the growth rate of endomorphisms.

The following provides an effective calculation for $\GR(\phi)$ (\cite{FFK}), which will be used throughout the paper. For an endomorphism $\phi: G\to G$, let $|\phi^k|$ denote the maximum of $|\phi^k(a_i)|$ over a generating set $A = \{a_1, \cdots, a_m\}$ of $G$, then
$$
\GR(\phi)=\lim_{k\to\infty}(|\phi^k|)^{1/k}
=\inf_{k}\left\{(|\phi^k|)^{1/k}\right\}.
$$

\section{Houghton's groups $\calH_n$}\label{sec:Houghton}

Let us use the following notational conventions. All bijections (or permutations) act on the right unless otherwise specified. Consequently $gh$ means $g$ followed by $h$. The conjugation by $g$ is denoted by $\mu(g)$, $h^g = g^{-1}hg =: \mu(g)(h)$,  and the commutator is defined by $[g,h] = ghg^{-1}h^{-1}$.

Our basic references are \cite{H,SR} for Houghton's groups and \cite{BCMR} for their automorphism groups.
Fix an integer $n\ge1$. For each $k$ with $1\le k\le n$, let
$$
R_k=\left\{me^{i\theta}\in\bbc\mid m\in\bbn,\ \theta=\tfrac{\pi}{2}+(k-1)\tfrac{2\pi}{n} \right\}
$$
and let $X_n=\bigcup_{k=1}^n R_k$ be the disjoint union of $n$ copies of $\bbn$,
each arranged along a ray emanating from the origin in the plane.
We shall use the notation $\{1,\cdots,n\}\x\bbn$ for $X_n$,
letting $(k,p)$ denote the point of $R_k$ with distance $p$ from the origin.

A bijection $g:X_n\to X_n$ is called an \emph{eventual translation}
if the following holds:
\begin{quote}
There exist an $n$-tuple $(m_1,\cdots,m_n)\in\bbz^n$ and a finite set $K_g\subset X_n$ such that
\begin{equation}\label{eq:eventual_translation}
(k,p)\cdot g:=(k,p+m_k)\quad \forall (k,p)\in X_n\setminus K_g.
\end{equation}
\end{quote}
An eventual translation acts as a translation on each ray outside a finite set.
For each $n \in \bbn$ the \emph{Houghton's group} $\calH_n$ is defined to be the group of all eventual translations of $X_n$.

Let $g_i$ be the translation on the ray of $R_1\cup R_i$ by $1$ for $2\le i\le n$.
Namely,
$$
(j,p)\cdot g_i=\begin{cases}
(1,p-1)&\text{if $j=1$ and $p\ge2$,}\\
(i,1)&\text{if $(j,p)=(1,1)$,}\\
(i,p+1)&\text{if $j=i$,}\\
(j,p)&\text{otherwise.}\end{cases}
$$


Johnson provided a finite presentation for $\calH_3$ in \cite{J} and the second author gave a finite presentation for $\calH_n$ with $n \geq 3$ in \cite{SR} as follows:
\begin{Thm}[{\cite[Theorem~C]{SR}}]\label{C}
For $n\ge3$, $\calH_n$ is generated by $g_2,\cdots,g_{n},\alpha$ with relations
\begin{equation*}
\alpha^2=1,\
(\alpha\alpha^{g_2})^3=1,\
[\alpha,\alpha^{g_2^2}]=1,\
\alpha=[g_i,g_j],\
\alpha^{g_i^{-1}}=\alpha^{g_j^{-1}}\
\end{equation*}
for $2\le i\ne j\le n$.
\end{Thm}

From the definition of Houghton's groups, the assignment $g\in\calH_n\mapsto (m_1,\cdots,m_n)\in\bbz^n$
defines a homomorphism $\pi=(\pi_1,\cdots,\pi_n):\calH_n\to\bbz^n$. Then we have:

\begin{Lemma}[{\cite[Lemma~2.3]{SR}}]
For $n\ge3$, we have $\ker\pi
=[\calH_n,\calH_n]$.
\end{Lemma}

Note that $\pi(g_i)\in\bbz^n$ has only two nonzero values $-1$ and $1$,
$$
\pi(g_i)=(-1,0,\cdots,0,1,0,\cdots,0)
$$
where $1$ occurs in the $i^{th}$ component.
Since the image of $\calH_n$ under $\pi$ is generated by those elements,
we have that
$$
\pi(\calH_n)=\left\{(m_1,\cdots,m_n)\in\bbz^n\mid \sum_{i=1}^n m_i=0\right\},
$$
which is isomorphic to the free Abelian group of rank $n-1$.
Consequently, $\calH_n$ ($n\ge3$) fits in the following short exact sequence
$$
1\lra\calH_n'=[\calH_n,\calH_n]\lra\calH_n\buildrel{\pi}\over\lra\bbz^{n-1}\lra1.
$$
The above abelianization, first observed by C. H. Houghton in \cite{H}, is the characteristic property of $\{\calH_n\}$ for which he introduced those groups in the same paper.
We may regard $\pi$ as a homomorphism $\calH_n\to\bbz^n\to\bbz^{n-1}$ given by
$$
\pi:g_i\mapsto(-1,0,\cdots,0,1,0,\cdots,0)\mapsto(0,\cdots,0,1,0,\cdots,0).
$$
In particular, $\pi(g_2),\cdots,\pi(g_n)$ form a set of free generators for $\bbz^{n-1}$.

As definition, $\calH_1$ is the symmetric group itself on $X_1$ with finite support,
which is not finitely generated.
Furthermore, $\calH_2$ is
$$
\calH_2=\langle g_2,\alpha \mid \alpha^2=1, (\alpha\alpha^{g_2})^3=1,
[\alpha,\alpha^{g_2^k}]=1 \text{ for all } |k|>1\rangle,
$$
which is finitely generated, but not finitely presented. It is not difficult to see that $\calH_2'=\FAlt_2$.

\begin{notation}
\begin{align*}
&\Sym_n= \text{ the full symmetric group of $X_n$,}\\
&\FSym_n= \text{ the symmetric group of $X_n$ with finite support,}\\
&\FAlt_n= \text{ the alternating group of $X_n$ with finite support.}
\end{align*}
\end{notation}
\noindent
\begin{Thm}[{\cite[Theorem~2.2]{BCMR}}]\label{auto}
For $n\ge2$, we have
$$
\aut(\calH_n)\cong\calH_n\rtimes \Sigma_n
$$
where $\Sigma_n$ is the symmetric group that permutes $n$ rays isometrically.
\end{Thm}

Here are some other known results for the Houghton's groups $\calH_n$:
\begin{itemize}
\item K.~S.~Brown(\cite{B}) showed that $\calH_n$ has type $F_{n-1}$ but not $F_n$.
\item By \cite[Theorem~2.18]{SR}, $\calH_n$ is an amenable group.
\item R\"{o}ver \cite{Rover} showed that for all $n\ge1, r\ge2,m\ge1$,
$\calH_n$ embeds in Higman's groups $G_{r,m}$ (defined in \cite{Hig}), and in particular all Houghton's groups are subgroups of Thompson's group $V$.
\end{itemize}

\section{Growth rates for automorphisms of $\calH_n$}\label{sec:auto}

In this section, we will study the growth rates for automorphisms of the Houghton's groups $\calH_n$.
We refer to \cite{B78, FJL} for general information on growth rates for endomorphisms
of finitely generated groups. Since $\calH_1=\FSym_1$ is not finitely generated, we must assume $n\ge2$.

For $n\ge3$, we fix a generating set $\Gamma =  \{g_{2},\cdots,g_{n}\}$ for $\calH_n$. Let $|g|$ denote the word length $|g|_\Gamma$ of $g\in \calH_n$ with respect to $\Gamma$. With $g_1 =\mathrm{id}$, for any automorphism $\phi=\mu(\sigma\gamma)$ we have
\begin{align*}
&g_{i}^\sigma=g_{\sigma(1)}^{-1}g_{\sigma(i)} \text{ for $i=2,\cdots,n$},\\
&g_{i}^\gamma=g_{i}^*= \text{ an eventual $g_{i}$}.
\end{align*}
The second line follows from the fact that $\gamma$ is a word of $g_{j}^{\pm1}$'s
and an observation that $g_{j}^{-1}g_{i}g_{j}$ is the translation on $R_1\cup\{(1)\}\cup R_i$ by $1$.
Observe also that
\begin{enumerate}
\item $(g_{i}^*)^{-1}=(g_{i}^{-1})^*$,\ $(g_{i}^{\pm1}g_{j})^*=(g_{i}^*)^{\pm1}g_{j}^*$,
\item $(g_{i}^*)^\sigma=(g_{i}^\sigma)^*$,\ $(g_{i}^*)^\gamma=(g_{i}^\gamma)^*=g_{i}^*$.
\end{enumerate}
In particular, we have $\phi(g_{i})=\phi(g_{i}^*)=\phi(g_{i})^*$.
Thus we have
\begin{align*}
\phi^m(g_{i}^*)={g_{\sigma^m(1)}^*}^{\!\!\!-1}g_{\sigma^m(i)}^*.
\end{align*}
This shows that if $\sigma$ is of order $m$, then $\phi^m(g_{i})=\phi^m(g_{i}^*)=g_{i}^*$.
Remark also that if $\gamma=1$, then $\phi^{m}(g_{i})=g_{i}$ for all $2\le i\le n$, i.e., $\phi^m=\id$.
Hence $\GR(\phi^m)=1$ and so $\GR(\phi)=1$. In fact, we have:

\begin{Thm}
For any automorphism $\phi$ of $\calH_n$ $(n\ge3)$, we have $\GR(\phi)=1$.
\end{Thm}

\begin{proof}
Let $|\gamma|=\ell$. Then we have $|\mu(\gamma)^k|\le 1+2\ell k$.
Indeed, since $\mu(\gamma)^k(g_{i})=\gamma^{-k}g_{i}\gamma^k$, we have
\begin{align*}
|\mu(\gamma)^k(g_{i})|&\le |\gamma^{-k}|+|g_i|+|\gamma^{k}|\\
&=1+2|\gamma^k|\le 1+2k|\gamma|=1+2k\ell.
\end{align*}
Thus we have $|\mu(\gamma)^k|\le 1+2k\ell$.
This implies that
$$
\GR(\mu(\gamma))=\lim_{k\to\infty}|\mu(\gamma)^k|^{1/k}
\le\lim_{k\to\infty}(1+2\ell k)^{1/k}=1.
$$

Let $\phi=\mu(\sigma\gamma)$ with $\sigma$ of order $m$.
Then
$$
\phi^m=\mu(\sigma\gamma)^m=\mu(\sigma^m\gamma^{\sigma^{m-1}}\cdots\gamma^\sigma\gamma)
=\mu(\gamma^{\sigma^{m-1}}\cdots\gamma^\sigma\gamma).
$$
Consequently, we have that
$$
\GR(\phi^m)=\GR(\mu(\gamma^{\sigma^{m-1}}\cdots\gamma^\sigma\gamma))\le1.
$$

Since $\GR(\phi)^m=\GR(\phi^m)$ (this follows from definition), we have $\GR(\phi)\le1$.
If $\GR(\phi)<1$, then by \cite[Lemma~2.3]{FJL}, $\phi$ is an eventually trivial endomorphism.
This is impossible because $\phi$ is an automorphism, and so we must have $\GR(\phi)=1$.
\end{proof}

Next, we consider $\calH_2$.
By Theorem~\ref{auto} again, every automorphism $\phi$ of $\calH_2$ is induced from the conjugation
by some element of $\sigma\gamma\in\calH_2\Sigma_2\subset\Sym_2$.
We remark that the proof of the above theorem relies on only the fact that automorphisms are induced from conjugations. This makes possible to repeat the above proof verbatim for the set $\Gamma=\{g_2,\alpha\}$ of generators for $\calH_2$. Therefore we have:

\begin{Thm}
For any automorphism $\phi$ of $\calH_2$, we have $\GR(\phi)=1$.
\end{Thm}

\section{Endomorphisms of $\calH_n$ which are not monic}\label{sec:not monic}

We recall the following results.

\begin{Thm}[{\cite[Theorem~8.1A]{DM}}]\label{thm:dm}
Let $\Omega$ be any set with $|\Omega|>4$.
Then the nontrivial normal subgroups of $\Sym(\Omega)$ are precisely:
$\FAlt(\Omega)$ and the subgroups of the form
$\Sym(\Omega,c)$ with $\aleph_0\le c\le |\Omega|$.
Here,
$$
\Sym(\Omega,c):=\{x\in\Sym(\Omega)\mid|\supp(x)|<c\}.
$$
In particular, $\Sym(\Omega,\aleph_0)=\FSym(\Omega)$, and so the nontrivial normal subgroups of $\Sym_n$ are
precisely $\FAlt_n$ and $\FSym_n$.
\end{Thm}

\begin{Thm}[{\cite[11.3.3]{Scott}}]\label{thm:scott}
The nontrivial normal subgroups of $\Sym(\Omega,B)$ are the groups $\Sym(\Omega,D)$ with $\aleph_0\le D<B\le2^{|\Omega|}$ and $\FAlt(\Omega)$.
In particular, we have:
\begin{enumerate}
\item[$(1)$] The only nontrivial normal subgroup of $\FSym_n$ is $\FAlt_n$ $($when $B=\aleph_0$$)$.
\item[$(2)$] The only nontrivial normal subgroups of $\Sym_n$ are $\FSym_n$ and $\FAlt_n$ $($when $B=2^{\aleph_0}$$)$.
\end{enumerate}
\end{Thm}

\begin{Cor}\label{cor:normal subgroup}
If $N$ is a nontrivial normal subgroup of $\calH_n$, then
$N\bigcap\FSym_n$ is either $\FAlt_n$ or $\FSym_n$. In particular $N$ contains $\FAlt_n$.
\end{Cor}

\begin{proof}
If $N$ is a normal subgroup of $\calH_n$, then $N\bigcap\FSym_n$ is a normal subgroup of $\FSym_n$.
By Part (1) of Theorem~\ref{thm:scott}, $N\bigcap\FSym_n$ is either $1$, $\FAlt_n$ or $\FSym_n$.

If $N\bigcap\FSym_n=\FAlt_n$, then $\FAlt_n\subset N$;
if $N\bigcap\FSym_n=\FSym_n$, then $\FAlt_n\subset\FSym_n\subset N$.

We now consider the case $N\bigcap\FSym_n=1$. This implies that the surjection $\calH_n\to\bbz^{n-1}$ maps
$N$ isomorphically into $\bbz^{n-1}$.
If $N\ne 1$, then $N$ contains an element $g\in\calH_n$ of infinite order. There exists a point $x\in\supp(g)$ whose orbit under $g$ is infinite.
Observe that for any transposition $\beta$ exchanging two points in the orbit of $x$, $[\beta,g]=g^\beta g^{-1}\in N$ is a $3$-cycle or a product of two commuting transpositions.
In fact, if $\beta$ exchanges two consecutive points $x$ and $(x)g$ in $\supp(g)$,
then $[\beta,g]=((x)g^{-1},x,(x)g)$ is a $3$-cycle;
if $\beta$ exchanges two non-consecutive points $x$ and $y$ in $\supp(g)$,
then we can show easily that $[\beta,g]=(x,y)((x)g^{-1},(y)g^{-1})$ is a product of two commuting transpositions.
In any case, $N$ contains a non-trivial torsion element, and so $N\bigcap \FSym_n \neq 1$. Consequently, $N=1$ if $N\bigcap\FSym_n=1$. This completes the proof.
\end{proof}

\begin{Cor}[\cite{CGP}]
The Houghton's group $\calH_n$ is not residually finite.
\end{Cor}

\begin{proof}
Suppose that $\calH_n$ is residually finite. By definition, for any $x\in\calH_n-\{1\}$
there exists a homomorphism $\phi$ from $\calH_n$ to a finite group such that $\phi(x)\ne1$.
Equivalently, for any $x\in\calH_n-\{1\}$ there exists a finite index normal subgroup $N_x$ of $\calH_n$ such that $x\notin N_x$. Since $N_x\ne1$, Corollary~\ref{cor:normal subgroup} implies that $\FAlt_n\subset N_x$ for all $x\in\calH_n-\{1\}$; in particular, for $x\in\FAlt_n-\{1\}$. This is a contradiction.
\end{proof}

\begin{Rmk}
The Houghton's group $\calH_n$ is not co-Hopfian for $n\geq 3$ (\cite{BCMR}).
That is, there is an injection which is not an isomorphism.
Indeed, $\phi: g_i\mapsto g_i^2$, $2\leq i\leq n$, defines such an injection. It is direct to check that $\phi(r)=1$ for all relators $r$ of the presentation in Theorem \ref{C}. For the injectivity of $\phi$ one can realize $\phi(\calH_n)$ as disjoint union of two subgroups each of which is isomorphic to $\calH_n$. (See Example \ref{ex:mono})
\end{Rmk}
\begin{Rmk}\label{rmk:central extension}
Recall that for a short exact sequence of groups $0\to A\to G\to Q\to 1$ with $A$ abelian, the following are equivalent:
\begin{enumerate}
\item $Q$ acts trivially on $A$.
\item $A$ lies in the center of $G$.
\end{enumerate}
In this case, we say that the short exact sequence is \emph{central}. Remark that if $A\cong\bbz_2$ then the exact sequence is always central. In fact, for any $g\in G$ and $a\ne1$ in $A$, we must have that $gag^{-1}=a$.
\end{Rmk}

\begin{Cor}[\cite{CGP}]\label{cor:Hopfian}
The Houghton's group $\calH_n$ is Hopfian.
\end{Cor}

\begin{proof}
Let $\phi:\calH_n\to\calH_n$ be an epimorphism, but not an isomorphism.
Put $K=\ker(\phi)$; then $K\ne1$, $\calH_n/K\cong\calH_n$
and from the proof of Corollary~\ref{cor:normal subgroup}, $K\bigcap\FSym_n$ is either $\FAlt_n$ or $\FSym_n$.

If $K\bigcap\FSym_n=\FSym_n$, then $\FSym_n\subset K$ and thus
we have
$$
\calH_n/\FSym_n(\cong\bbz^{n-1})\ -\!\!\!\twoheadrightarrow \calH_n/K\overset\cong\lra\calH_n.
$$
This implies that $\calH_n$ is abelian, which is absurd.

Consider next the case $K\bigcap\FSym_n=\FAlt_n$. Then we have a commuting diagram
$$
\CD
@.1@.1@.1\\
@.@AAA@AAA@AAA\\
1@>>>\FSym_n/K\cap \FSym_n@>>>\calH_n/K@>>>\calH_n/K\cdot\FSym_n@>>>1\\
@.@AAA@AAA@AAA\\
1@>>>\FSym_n@>>>\calH_n@>>>\bbz^{n-1}@>>>1\\
@.@AAA@AAA@AAA\\
1@>>>K\cap \FSym_n@>>>K@>>>K/K\cap \FSym_n@>>>1\\
@.@AAA@AAA@AAA\\
@.1@.1@.1
\endCD
$$
Since $\FSym_n/K\cap \FSym_n=\FSym_n/\FAlt_n\cong\bbz_2$, the top row exact sequence is a central extension by Remark~\ref{rmk:central extension}. This implies that $\calH_n/K\cong\calH_n$ has a nontrivial center.
But Lemma~\ref{lemma:centerless} below shows that $\calH_n$ has trivial center.
\end{proof}

\begin{Lemma}\label{lemma:centerless}
For all $n$, the center of $\calH_n$ is trivial.
\end{Lemma}
\begin{proof}
Assume that $\calZ(\calH_n)$ is nontrivial.
By Corollary \ref{cor:normal subgroup}, $\FAlt_n\subset\calZ(\calH_n)$. This implies that $\FAlt_n$ is abelian, a contradiction.
\end{proof}

Let $\varphi$ be an endomorphism of $\calH_n$.
In this section, we shall assume $K=\ker(\varphi)\ne1$.
[Since $\calH_n$ is Hopfian by Corollary~\ref{cor:Hopfian}, $\varphi$ cannot be epic.]
By Corollary~\ref{cor:normal subgroup}, $K\cap\FSym_n$ is either $\FAlt_n$ or $\FSym_n$.

Consider first the case $K\cap\FSym_n=\FSym_n$.
Since $\FSym_n\subset K$,
$\varphi$ induces the following diagram
$$
\CD
\calH_n@>>>\calH_n/\FSym_n@>>>\calH_n/K\\
@VV{\varphi}V@VV{\bar\varphi}V@VV{\hat\varphi}V\\
\calH_n@>>>\calH_n/\FSym_n@>>>\calH_n/K
\endCD
$$
where the horizontal maps are canonical surjections.
It is immediate from the definition of the growth rate that
$$
\GR(\hat\varphi)\le\GR(\bar\varphi)\le\GR(\varphi).
$$
We claim now that $\GR(\hat\varphi)=\GR(\bar\varphi)=\GR(\varphi)$.
This follows from the fact that the endomorphism $\hat\varphi:\calH_n/K\to\calH_n/K$ induced by $\varphi$
is simply the restriction of $\varphi$ on the $\varphi$-invariant subgroup $\varphi(\calH_n)$ of $\calH_n$,
$\varphi|_{\varphi(\calH_n)}:\varphi(\calH_n)\to\varphi(\calH_n)$. It is known that
$\GR(\varphi)=\GR(\varphi|_{\varphi(\calH_n)})=\GR(\hat\varphi)$, see for example \cite[Sect.~4]{Canovas}.
Therefore, we have
$$
\GR(\varphi)=\GR(\bar\varphi).
$$
Recall also that $\GR(\bar\varphi)$ is the spectral radius of the integer matrix determined by
the endomorphism $\bar\varphi$ of $\calH_n/\FSym_n=\bbz^{n-1}$.

Next we consider the case $K\bigcap\FSym_n=\FAlt_n$. Since $\FAlt_n\subset K$,
$\varphi$ induces the following diagram
$$
\CD
\calH_n@>>>\calH_n/\FAlt_n@>>>\calH_n/K\\
@VV{\varphi}V@VV{\tilde\varphi}V@VV{\hat\varphi}V\\
\calH_n@>>>\calH_n/\FAlt_n@>>>\calH_n/K
\endCD
$$
where the horizontal maps are canonical surjections.
Repeating the same argument as above, we have
$$
\GR(\varphi)=\GR(\tilde\varphi).
$$
Since $\FSym_n$ is the subgroup of $\calH_n$ consisting of all elements that have finite order, $\FSym_n$ is $\phi$-invariant.
Consequently, $\phi$ induces the following diagram
$$
\CD
\FSym_n/\FAlt_n@>>>\calH_n/\FAlt_n@>>>\calH_n/\FSym_n\\
@VV{\phi'}V@VV{\tilde\phi}V@VV{\bar\phi}V\\
\FSym_n/\FAlt_n@>>>\calH_n/\FAlt_n@>>>\calH_n/\FSym_n
\endCD
$$
By \cite[Lemma~2.5]{FJL}, we have
$$
\GR(\tilde\phi)=\max\left\{\GR(\phi'),\GR(\bar\phi)\right\}.
$$
Remark also that $\FSym_n/\FAlt_n=\bbz_2$, hence $\GR(\phi')=0$ or $1$ depending on $\phi'$ is trivial or not trivial (then $\phi'$ is the identity).
If $\GR(\bar\phi) =\sp[\bar\phi]<1$ then $\bar\phi$ is eventually trivial or  $\GR(\bar\phi)=0$. This implies that $\GR(\bar\phi) \geq 1$ if $\bar \phi$ is not eventually trivial. So we can summarize our discussion in this section as follows.

\begin{Thm}
Let $\phi$ be an endomorphism of $\calH_n$ with nontrivial kernel.
Then $\FSym_n$ is $\phi$-invariant and so there results in the following commutative diagram
$$
\CD
\FSym_n/\FAlt_n@>>>\calH_n/\FAlt_n@>>>\calH_n/\FSym_n\\
@VV{\phi'}V@VV{\tilde\phi}V@VV{\bar\phi}V\\
\FSym_n/\FAlt_n@>>>\calH_n/\FAlt_n@>>>\calH_n/\FSym_n
\endCD
$$
If $\phi'$ is nontrivial and $\bar\phi$ is eventually trivial then $\GR(\phi)=1$, and otherwise we have
$$
\GR(\phi)
= \sp[\bar\phi].
$$
\end{Thm}

\section{Monomorphisms of $\calH_n$}

Let $2\leq n$ and let $\phi$ denote a monomorphism of $\calH_n$ throughout this section unless otherwise stated. Let {$\supp g$} denote the support of $g\in \calH_n$.
Recall from Theorem \ref{C} that if $n\ge3$ then $\calH_n$ has the following presentation
\begin{equation}\label{eq:relations}
\alpha^2=1,\
(\alpha\alpha^{g_2})^3=1,\
[\alpha,\alpha^{g_2^2}]=1,\
\alpha=[g_i,g_j],\
\alpha^{ g_i^{-1}}=\alpha^{ g_j^{-1}}\
\end{equation} for $2\le i\ne j\le n$.

For $i\neq j$ and $k\in \bbn$, the support of the transposition $\alpha^{g_i^{k+1}}$ belongs to $R_{i}$ and does not intersect $\supp g_j$. Consequently we have the following identities
\begin{equation}\label{eq:r_6}
[\alpha^{g_i^{k+1}}, g_j]=1
\end{equation}
for $k\in\bbn$ and $i\neq j$. By the same reason, we also have
\begin{equation}\label{eq:r_8}
[\alpha, \alpha^{g_i^{k+1}}]=1 
\end{equation}for all $k \in \bbn$.
Note that the actions of $g_i^{-1}$ and $g_j^{-1}$ coincide on $R_1$. So we have identities
\begin{equation}\label{eq:r_7}
\alpha^{g_i^{-k}}= \alpha^{g_j^{-k}} 
\end{equation}
for all $2\leq i,j\leq n$ and $k \in \bbn$.

Note that any homomorphism of $\calH_n$ must preserve all identities $r_1$ through $r_5$ as well as identities (\ref{eq:r_6}), (\ref{eq:r_7}) and (\ref{eq:r_8}).\\

We say $P\in \supp g$ is an \emph{essential point} of $g$ if its orbit under $g$ is infinite. Let $\Es g$ denote the set of essential points of $g$.
For two elements $g,h\in \calH_n$, we say $g$ \emph{intersects} $h$ if $\supp g \cap \supp h \neq \emptyset$.
Similarly we say $g$ \emph{intersects} $X \subset X_n$ if $\supp g \cap X\neq \emptyset$.

For each pair $2\leq i \neq j\leq n$, $[\phi(g_i),\phi(g_j)]=\phi(\alpha)$ is nontrivial and has order $2$ for a monomorphism $\phi$. So $\phi(\alpha)$ can be written
\begin{equation}\label{eq:cycle_alpha}
\phi(\alpha) = \tau_1 \cdots \tau_\ell
\end{equation} as a product of commuting transpositions. Let $T_\alpha$ denote the set of transpositions in (\ref{eq:cycle_alpha}), and let $S_\alpha= \supp \phi(\alpha)$. Note that each $P\in S_\alpha$ determines a \emph{unique} transposition in $T_\alpha$ which moves $P$. Let $T_i$ be the collection of transpositions $\tau$ in (\ref{eq:cycle_alpha}) with $\supp \tau \cap\Es \phi(g_i)\neq \emptyset$, $2\leq i \leq n$.

\begin{Lemma}\label{lemma:transposition_on_ess_orbit}
Suppose $\phi$ is a monomorphism of $\calH_n$ with $n\geq 2$. For each $i$, $T_i \neq \emptyset$.
\end{Lemma}
\begin{proof}
We first consider a monomorphism when $n\ge 3$. Fix $i$. The conjugations $\alpha^{g_i^k}$ are all distinct for $k \in \bbz$. A monomorphism $\phi$ induces a bijection between infinite sets $\{\alpha^{g_i^k}\,|\, k \in \bbz\} \leftrightarrow \{\phi(\alpha^{g_i^k})\,|\, k \in \bbz\}=A$. Let us denote by $\beta_k$ conjugation
$$
\beta_k = \phi(\alpha)^{\phi(g_i)^k}=(\tau_1 \cdots \tau_\ell)^{\phi(g_i)^k}.
$$
If all points of $\supp \phi(\alpha)$ are not essential points of $\phi(g_i)$, then the set $A$ must be finite since $\phi(g_i)^{k_0}$ conjugates $\phi(\alpha)$ to itself for some integer $k_0\neq 0$.
Suppose that all points of $S_\alpha$ are not essential points of $\phi(g_i)$. If $S_\alpha\bigcap \supp\phi(g_i)=\emptyset$, then it is clear that $\phi(\alpha)^{\phi(g_i)}=\phi(\alpha)$. Next consider $S_\alpha\bigcap \supp\phi(g_i)=\{P_1,\cdots,P_s\}\ne\emptyset$.
By the supposition, each $P_j$ is not an essential point of $\phi(g_i)$, i.e., $(P_j)\phi(g_i)^{k_j}=P_j$ for some $k_j>0$. Let $k_0=\mathrm{lcm}(k_1,\cdots,k_s)$. Hence $(P_j)\phi(g_i)^{k_0}=P_j$ for all $j=1,\cdots,s$ and so
$(P)\phi(g_i)^{k_0}=P$ for all $P\in S_\alpha$. This shows that $\beta_{k_0}=\phi(\alpha)$. 

Let $P\in \supp \phi(\alpha)$ be an essential point of $\phi(g_i)$. Let $\tau'$ denote the unique transposition in (\ref{eq:cycle_alpha}) which moves $P$ to $Q=(P)\tau'$. We claim that $Q$ is also an essential point of $\phi(g_i)$. Assume the contrary that
$$(Q)\phi(g_i)^{k_0}=Q
$$ for some integer $k_0 \neq 0$ ($k_0=1$ when $Q$ is fixed by $\phi(g_i)$, and $|k_0| \geq 2$ when $Q\in \supp \phi(g_i)$). We need to examine $\beta_k$ when $k$ is a multiple of $k_0$.

Since $\phi(g_i)^{mk_0}$ fixes $Q$ for each nonzero $m\in \bbz$, $\phi(g_i)^{mk_0}$ conjugates $\tau' = (P, Q)$ to the transposition, denoted by $\tau'_m$, which exchanges $(P)\phi(g_i)^{mk_0}$ and $Q$.
See Figure \ref{fig:on_p}. Observe that
$$
Q \in \supp\tau' \cap \supp\tau'_m \subset S_\alpha\cap\supp \beta_{mk_0}
$$
for each nonzero $m\in \bbz$. To draw a contradiction, we use identities (\ref{eq:r_8}) which say in particular that $\phi(\alpha)$ commute with $\beta_k$ for all $k=mk_0$ with $|mk_0|\geq 2$. We apply Lemma \ref{lemma:comm_involution} to see that the  intersection $I_m=S_\alpha \cap \supp \beta_{mk_0}$ satisfies that
\begin{equation}
(I_m)\phi(\alpha)=I_m=(I_m)\beta_{mk_0}
\end{equation} for all $m$ with $|mk_0|\geq 2$. In particular, $S_\alpha$ contains $(Q)\beta_{mk_0}=(Q)\tau'_m =(P)\phi(g_i)^{mk_0}$ for infinitely many $m\in \bbz$. 
Since $P$ is an essential point of $\phi(g_i)$, $S_\alpha$ must be an infinite set. This contradicts that an element $\phi(\alpha) \in\FSym_n$ has a finite support. Therefore $Q$ is also an essential point of $\phi(g_i)$, and hence the transposition $\tau'$ exchanges two essential points $P$ and $Q$ of $\phi(g_i)$.

For $n=2$, recall that $\calH_2$ has a presentation
\begin{equation}\label{eq:presentation H_2}
\calH_2=\langle g_2,\alpha \mid \alpha^2=1, (\alpha\alpha^{g_2})^3=1,
[\alpha,\alpha^{g_2^k}]=1 \text{ for all } |k|>1\rangle
\end{equation} It follows from analogous arguments applied to $g_2$ with commutation relations above that $T_i \neq \emptyset $.
\end{proof}
\begin{figure}[h]
    \includegraphics[width=.27\textwidth]{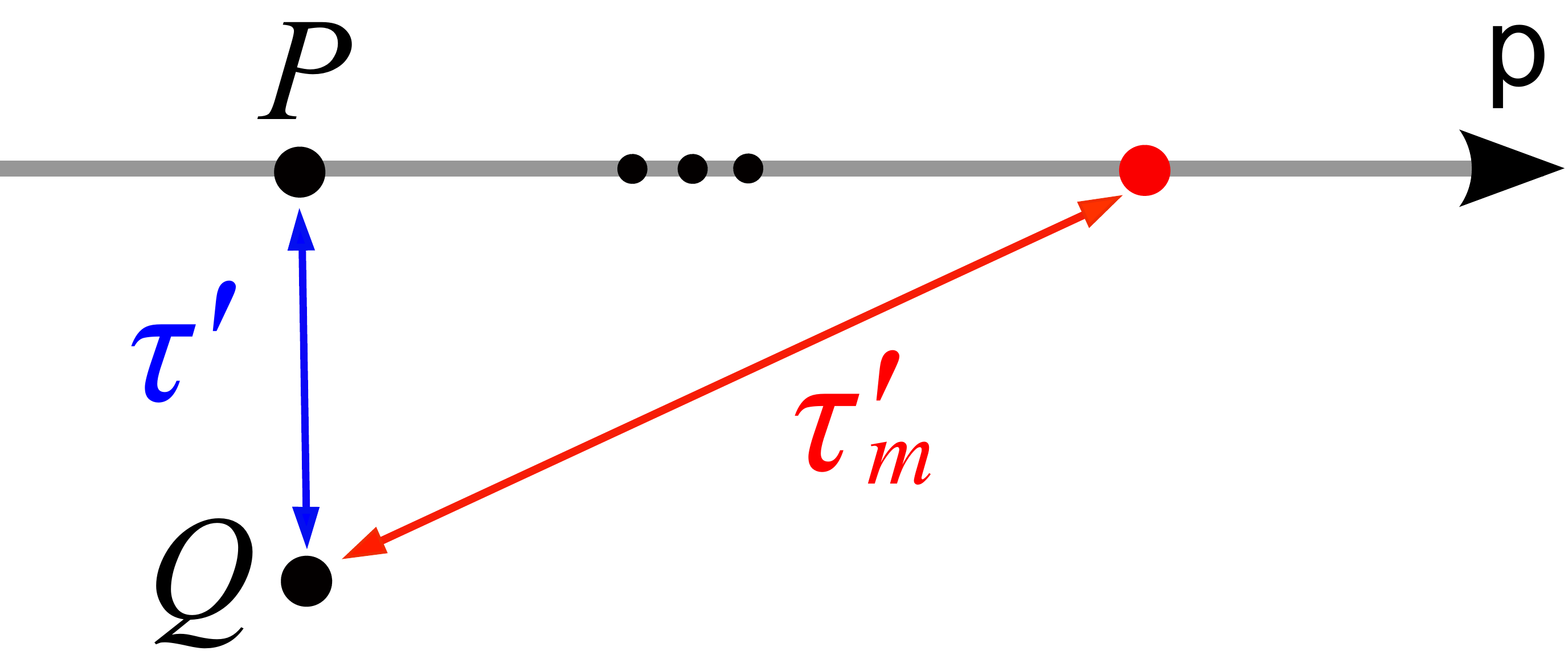}
    \caption{$\tau'_m$ moves $Q$ to $(P)\phi(g_i)^{mk_0}$ for infinitely many $m\in \bbz$.}
    \label{fig:on_p}
\end{figure}

\begin{Lemma}\label{lemma:comm_involution}
Suppose two involutions $\beta, \gamma \in \FSym_n$ commute. Then $I=(I)\beta = (I)\gamma$ where $I=\supp\beta \cap \supp \gamma$.
\end{Lemma}
\begin{proof}
We may assume the intersection $I =\supp\beta \cap \supp \gamma$ is non-empty. Since $\beta$ is an involution, $B=I\cup (I)\beta$ is a $\beta$-invariant subset of $\supp\beta$. Similarly $C =I\cup (I)\gamma$ satisfies $(C)\gamma=C$. Let $\beta_B$ and $\gamma_C$ denote the restrictions of $\beta$ and $\gamma$ on $B$ and $C$ respectively. We denote by $\beta_{B^c}$ and $\gamma_{C^c}$ the restrictions on the complements $B^c=\supp \beta \setminus B$ and $C^c = \supp \gamma \setminus C$ respectively. We can write
$$
\beta = \beta_{B}\cdot\beta_{B^c} \text{  and  }  \gamma = \gamma_{C}\cdot\gamma_{C^c}.
$$
Note that $ B^c \cap \supp \gamma=\emptyset$ and $C^c \cap \supp \beta=\emptyset$.

Since $\beta^\gamma = \beta$ and
$$
\beta^\gamma=(\beta_B \cdot \beta_{B^c} )^\gamma = (\beta_B)^\gamma  \cdot (\beta_{B^c})^\gamma =(\beta_B) ^\gamma \cdot \beta_{B^c} =(\beta_B) ^{\gamma_C} \cdot \beta_{B^c},
$$
we have $(\beta_B) ^{\gamma_C} \cdot \beta_{B^c} =\beta=\beta_B\cdot \beta_{B^c}$.
So, $\gamma_C$ conjugates $\beta_B$ to itself, and $B$ is an invariant set of $\gamma_C$.  Similarly $\gamma^\beta = \gamma$ implies that $(\gamma_C)^{\beta_B}=\gamma_C$.

We first show $B\subset C$. If $P \in B\setminus C \subset(I)\beta\setminus (I)\gamma$, there exists $Q \in I$ such that $P=(Q)\beta_B$. Observe that $P$ and $Q$ determine a unique transposition in the cycle decomposition of $\beta_B$ such that
$$
P=(Q)\tau.
$$
Since $\gamma_C$ moves $Q$ to $Q'\neq P$ while fixing $P\notin C$, $\gamma_C$ conjugates $\tau$ to $\tau^{\gamma_C}\neq \tau$ which exchanges $P$ and $Q'$. However this contradicts the uniqueness of the transposition $\tau$ of $\beta_B$. In other words $\gamma_C$ fails to conjugate $\beta_B$ to itself. Therefore $(I)\beta\subset(I)\gamma$ and so $B \subset C$.

For the converse $(I)\gamma\subset (I)\beta$ or $C \subset B$, one can use analogous arguments applied to $\gamma^\beta = \gamma$. In all,
$$
(I)\beta=(I)\gamma,\
B= I\cup (I)\beta = I\cup (I)\gamma =C.
$$
Furthermore, if $P\in (I)\beta=(I)\gamma$ then $P=(Q)\beta=(Q')\gamma$ for some $Q,Q'\in I$, thus $P\in\supp\beta\cap\supp\gamma=I$. The reverse inclusion follows since $\beta$ and $\gamma$ are involutions, hence $(I)\beta =(I)\gamma=I$.
\end{proof}

\begin{Rmk}\label{rmk:square_in_commuting_involutions}\textbf{(Squares on $I$)}
An interesting consequence of Lemma~\ref{lemma:comm_involution} is that
a point $P\in I$ with $(P)\beta \neq (P)\gamma$ determines a certain \emph{square}
(or \emph{$4$-cycle of transpositions}) on $I$.
Let $\tau_1$ and $\tau'_1$ be the transpositions of $\beta$ and $\gamma$ respectively
which moves $P$ to distinct points.
The Lemma~\ref{lemma:comm_involution} says $Q=(P)\tau_1 \in I \subset\supp \gamma$
and $R=(P)\tau'_1 \in I \subset \supp \beta$.
Let $\tau_2$ and $\tau_2'$ be the unique transpositions of $\beta$ and $\gamma$ respectively
so that $R \in \supp \tau_2$ and $Q \in \supp \tau'_2$.
The intersection $\supp \tau_2 \cap \supp \tau_2'$ contains a point $S$
since $\beta\gamma =\gamma \beta$ applied to $P$ implies
\begin{align*}
&P \xrightarrow {\beta} (P)\beta = (P)\tau_1=Q \xrightarrow {\gamma} (Q)\gamma = (Q) \tau'_2 = S,\\
& P\xrightarrow {\gamma} (P)\gamma = (P)\tau'_1 =R\xrightarrow {\beta} (R)\beta = (R) \tau_2=S.
\end{align*}
Applying alternate compositions of $\beta$ and $\gamma$,
we have a square on the four points
$P\xrightarrow{\tau'_1}R\xrightarrow{\tau_2}S\xrightarrow{\tau'_2}Q
\xrightarrow{\tau_1}P$.
Reading off four transpositions of $\beta$ and $\gamma$ in order
we have a $4$-cycle $\tau'_1 -\tau_2 -\tau'_2-\tau_1$ of transpositions. (This should not be confused with a $4$-cycle of $\FSym_n$.) Figure \ref{fig:square} illustrates the square on four points and corresponding $4$-cycle of transpositions.

\begin{center}
\begin{figure}[h]
    \includegraphics[width=.18\textwidth]{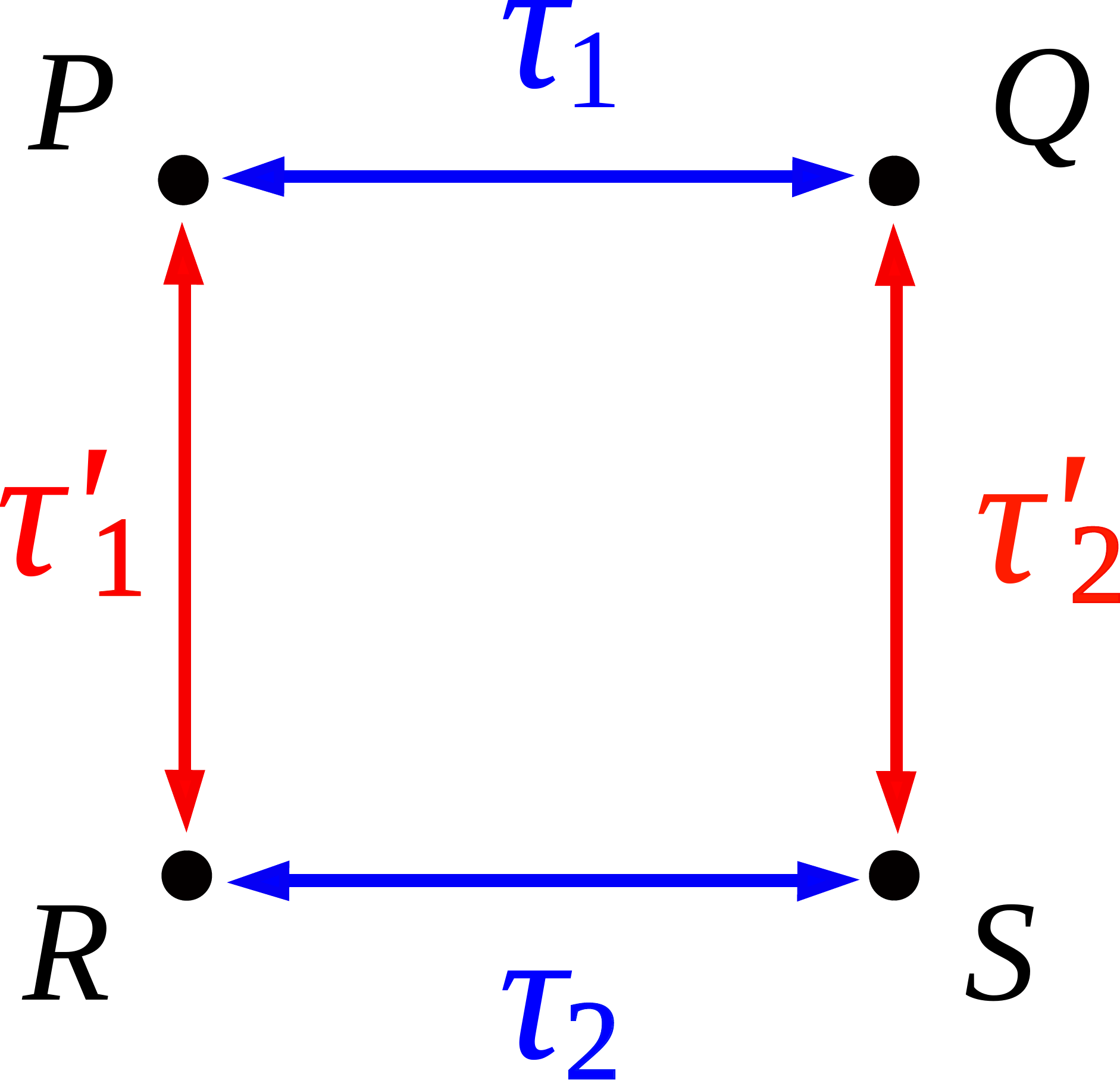}
    \caption{}
    \label{fig:square}
\end{figure}
\end{center}
\end{Rmk}

Next we consider a useful necessary condition when two involutions $\beta, \gamma\in \FSym_n$
form $\beta \gamma$ of order $3$.
Suppose $\beta$ and $\gamma$ have cycle decompositions
\begin{equation*}
\beta=\tau_1 \cdots \tau_k \text {  and  }\gamma=\tau'_1 \cdots \tau'_\ell.
\end{equation*}
Let $\calS=\{\tau_1, \cdots, \tau_k\}$ and $\calT= \{\tau'_1 , \cdots, \tau'_\ell\}$,
and let $\frakP(\calS)$ and $\frakP(\calT)$ denote the power sets of $\calS$ and $\calT$ respectively.
We define assignments
\begin{equation}\label{eq:xi and chi}
\xi:\calS \lra \frakP(\calT)\ \text{  and  }\ \chi:\calT\lra\frakP(\calS)
\end{equation}
so that $\supp \tau_i$ has non-trivial intersection with each support of $\tau'_j \in \xi(\tau_i)$.
The assignment $\chi(\tau_i)$ is defined by the same manner after swapping $\calS$ and $\calT$.
As we shall see below the cardinalities of $\xi(\ast)$ and $\chi(\ast)$ are either one or two.
Let $\calS_1 \subset \calS$ and $\calT_1 \subset \calT$ be defined by
\begin{equation*}
\calS_1= \{\tau_i \in \calS\,:\,| \xi(\tau_i)|=1\}, \quad \calT_1=\{\tau'_j \in \calT :| \chi(\tau'_j)|=1\}.
\end{equation*}

\begin{Lemma}\label{lemma:3-cycle} With the notations for two involutions $\beta$ and $\gamma$ such that $\beta \gamma$ has order $3$ as above, $1\leq |\xi(\tau_i)| \leq 2 $ and $1 \leq |\chi(\tau'_j)| \leq 2$ for each $\tau_i\in \calS$ and $\tau'_j \in \calT$. There exists a bijection $\calS_1 \leftrightarrow \calT_1$
where correspondence is given by $\xi(\tau_i) = \{\tau'_j\} \Leftrightarrow\chi(\tau'_j) = \{\tau_i\}$.
\end{Lemma}

\begin{proof}
Fix a transposition $\tau_i \in \calS$ which exchanges two points $P$ and $Q$.
It is obvious that $\calT$ contains at most two transpositions which move $P$ or $Q$.
So the cardinality of $\xi(\tau_i)$ is at most $2$.
To show $|\xi(\tau_i)|\neq 0$, assume that every transposition of $\calT$ fixes $P$ and $Q$.
For convenience we write
$\beta=\tau_i \cdot \beta'$ where $\beta'= \tau_1 \cdots \widehat{\tau_i} \cdots \tau_k$ fixes $P$ and $Q$.
If $\gamma$ fixes $P$ and $Q$ then the product $(\beta \gamma)^3$ can be written as
$$
(\beta \gamma)^3 = (\tau_i \beta' \cdot \gamma)^3 = \tau_i^3 (\beta'\gamma)^3 = \tau_i(\beta'\gamma)^3
$$
since $\tau_i$ commutes with $\beta'$ and $\gamma$.
The support of  $\tau_i(\beta'\gamma)^3$ contains $P$ and $Q$ and so it is not empty.
This means that $\beta \gamma$ is not an element of order $3$.
We have shown that $|\xi(\tau_i)| $ equals $1$ or $2$.
For the cardinality of $\chi(\ast)$ one applies analogous arguments after swapping $\beta$ and $\gamma$.

To establish the correspondence $\calS_1 \leftrightarrow \calT_1$,
suppose $\xi(\tau_i) = \{\tau'_j\}$, i.e., $\tau'_j \in \calT$ is the unique transposition which moves $P$ or $Q$.
If $\tau_i=\tau'_j$ it is obvious that $\xi(\tau_i) = \{\tau'_j\} \Leftrightarrow\chi(\tau'_j) = \{\tau_i\}$.
We assume $\tau_i\neq\tau'_j$ and show that $\chi(\tau'_j) = \{\tau_i\}$ as follows.
It suffices to show that $|\chi(\tau'_j)|=1$ since $\tau_i \in \chi(\tau'_j)$.
If $|\chi(\tau'_j)|=2$ then $\calS$ contains a transposition $\tau \neq\tau_i$,
whose support intersects $\supp \tau'_j$ nontrivially.
Since $\tau_j'$ has to move $P$ or $Q$, so we may assume $\tau'_j$ moves $P$. Let $P' = (P)\tau'_j$. Then $\tau$ exchanges $P'$ and $R=(P')\tau$, and the point $Q$ is fixed by $\gamma$.
Applying the product $\beta \gamma$ repeatedly to the point $R$ we have $(R)(\beta\gamma)^3 = P'\neq R$ since
\begin{align*}
&R \xrightarrow {\beta} (R)\beta =(R)\tau=P'  \xrightarrow {\gamma} (P')\gamma = (P')\tau'_j = P,\\
&P \xrightarrow {\beta} (P)\beta = (P)\tau_i = Q \xrightarrow {\gamma} (Q) \gamma =Q,\\
&Q\xrightarrow {\beta} (Q)\beta = (Q)\tau_i = P \xrightarrow {\gamma} (P) \gamma =(P)\tau'_j = P'.
\end{align*}So $R\in\supp (\beta\gamma)^3\neq \emptyset$. This contradicts that the product $\beta \gamma$ has order $3$. Figure \ref{figure:1-1} illustrates this situation.

So far we have shown that $\xi(\tau_i) = \{\tau'_j\}$ implies $\chi(\tau'_j) = \{\tau_i\}$. The converse follows because a symmetric argument shows that $\gamma \beta = (\beta \gamma)^{-1}$ fails to have  order $3$ if $\chi(\tau'_j) =\{\tau_i\}$ but $|\xi(\tau_i)|=2$. Therefore $\xi(\tau_i) = \{\tau'_j\} \Leftrightarrow\chi(\tau'_j) = \{\tau_i\}$ determines the correspondence $\calS_1 \leftrightarrow \calT_1$.
\end{proof}

\begin{figure}[h]
    \includegraphics[width=.16\textwidth]{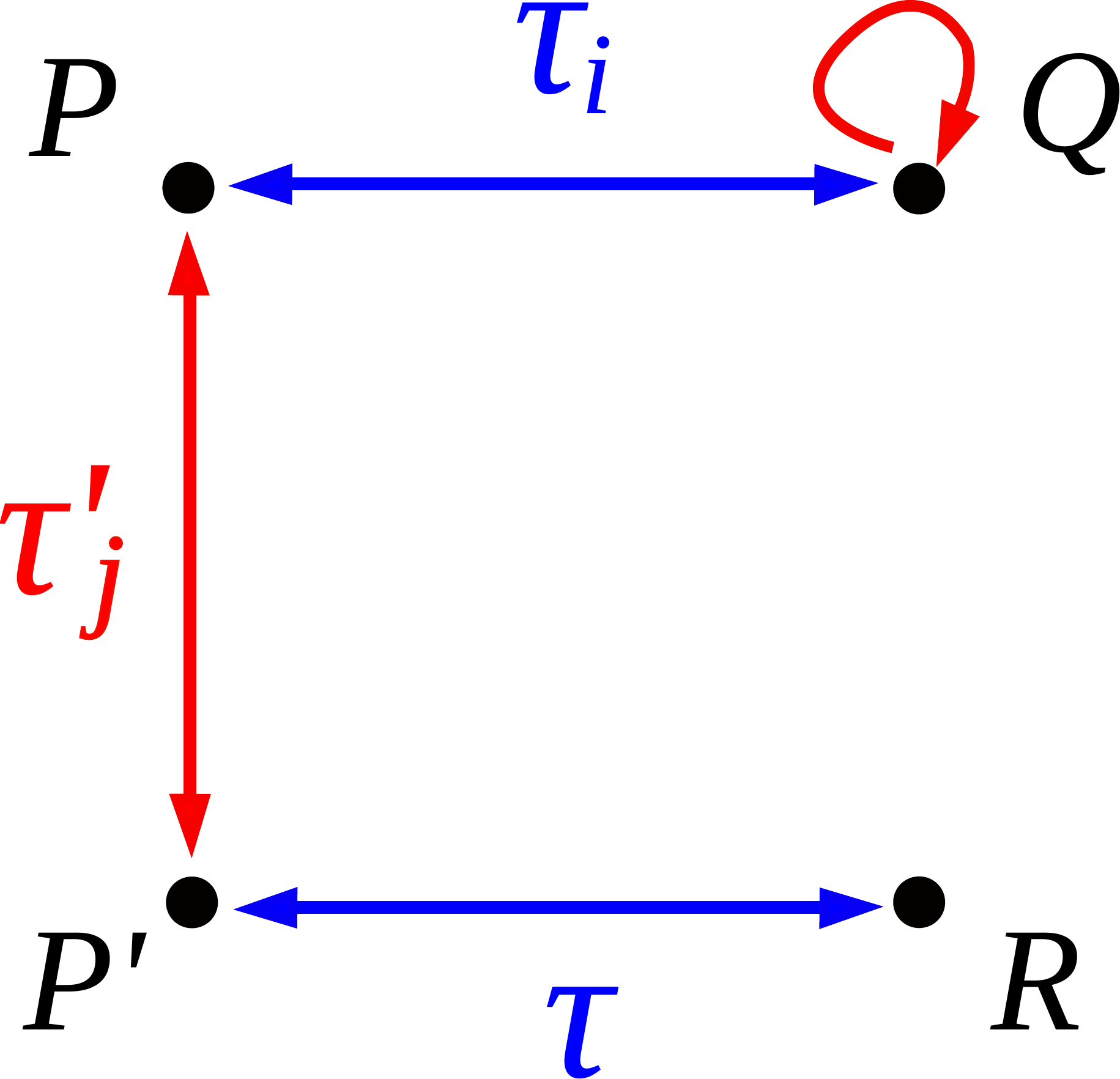}
    \caption{}\label{figure:1-1}
\end{figure}

\begin{Rmk}\label{rmk:hexagon}\textbf{(Hexagons on $I$)}
To understand the relation between complements of $\calS_1$ and $\calT_1$ consider assignments $\xi$ and $\chi$ on
$$
\calS_2 = \calS\setminus \calS_1\text{  and   } \calT_2 = \calT\setminus \calT_1.
$$ Lemma \ref{lemma:3-cycle} implies that $|\xi(\tau_i)|=2$ for all $\tau_i \in \calS_2$, and conversely $|\chi(\tau'_j)|=2 $ for all $\tau'_j \in \calT_2$.
It is interesting to see that each element $\tau_i \in \calS_2$ involves a certain \emph{hexagon} (or $6$-cycle of transpositions) on $I= \supp \beta \cap \supp \gamma$. Suppose $\tau_i\in \calS_2$ exchanges $P$ and $Q$. Let us describe the other $5$ transpositions. Let $\xi(\tau_i) =\{\tau'_P, \tau'_Q\}$ where $\tau'_P$ exchanges $P$ and $(P)\tau'_P = P_1$ and $\tau'_Q$ exchanges $Q$ and $(Q)\tau'_Q = Q_1$. Since $|\chi(\tau'_P)|=2$ and $\tau_i\in\chi(\tau'_P)$, $\chi(\tau'_P)$ contains $\tau_{P_1}\in \calS_2$ which exchanges $P_1$ and $P_2 =(P_1) \tau_{P_1}$. Similarly $\chi(\tau'_Q)$ contains two transpositions $\tau_i$ and $\tau_{Q_1}\in \calS_2$ which exchanges $Q_1$ and $Q_2 =(Q_1) \tau_{Q_1}$. One can check that $\tau_{P_1} \neq \tau_{Q_1}$. If $\tau_{P_1} = \tau_{Q_1}$ the four transpositions $\tau_i$, $\tau'_Q$, $\tau_{P_1}=\tau_{Q_1}$, and $\tau'_P$ form a square as described in Remark \ref{rmk:square_in_commuting_involutions}. It is direct to see that $\{P_1, Q_1\} \subset \supp (\beta \gamma)^3$, which contradicts that $\beta \gamma$ has order $3$. So far we have $5$ transpositions $\tau_i, \tau'_P, \tau'_{Q'}, \tau_{P_1}$, and $\tau_{Q_1}$  . Next we check that $\tau_{P_1}$ and $\tau_{Q_1}$ determine a unique transposition $\tau'_j \in \calT_2$ so that $\tau'_j \in \xi(\tau_{P_1})\cap \xi(\tau_{Q_1})$, namely $\tau'_j$ exchanges $P_2$ and $Q_2$.
Applying $\beta\gamma$ to the point $Q_2$ twice we have
\begin{align*}
&Q_2 \xrightarrow {\beta} (Q_2)\beta = (Q_2)\tau_{Q_1} = Q_1 \xrightarrow {\gamma} (Q_1)\gamma = (Q_1)\tau'_Q=Q\\
&Q \xrightarrow {\beta} (Q)\beta=(Q)\tau_i = P\xrightarrow {\gamma} (P)\gamma=(P)\tau'_P =P_1.\\
\end{align*}Since $\beta \gamma$ has order $3$ we must have $(Q_2)(\beta\gamma)^{3}=(P_1)\beta\gamma= Q_2$. The later identity means that
$$
(P_1)\tau_{P_1}\gamma =(P_2)\gamma = (P_2)\tau'_j=Q_2.
$$
Thus the transposition $\tau'_j$ exchanges $P_2$ and $Q_2$, and so it is the unique transposition of $\calT_2$ such that $\tau'_j \in \xi(\tau_{P_1})\cap \xi(\tau_{Q_1})$ as desired. Alternate composition $\beta$-and-then-$\gamma$ applied to $P$ determines $6$ points
$$
P\xrightarrow{\tau_i}Q\xrightarrow{\tau'_Q}Q_1\xrightarrow{\tau_{Q_1}}Q_2
\xrightarrow{\tau'_j}P_2\xrightarrow{\tau_{P_1}}P_1\xrightarrow{\tau'_{P'}}P.
$$ The six transpositions in order form a hexagon or $6$-cycle of transpositions: $\tau_i-\tau'_Q-\tau_{Q_1}-\tau'_j-\tau_{P_1}-\tau'_P$ (Again, one should distinguish this from a $6$-cycle of $\calH_n$.) Figure \ref{fig:hexagon} illustrates the hexagon on four points and corresponding $6$-cycle of transpositions.
\begin{figure}[h]
    \includegraphics[width=.2\textwidth]{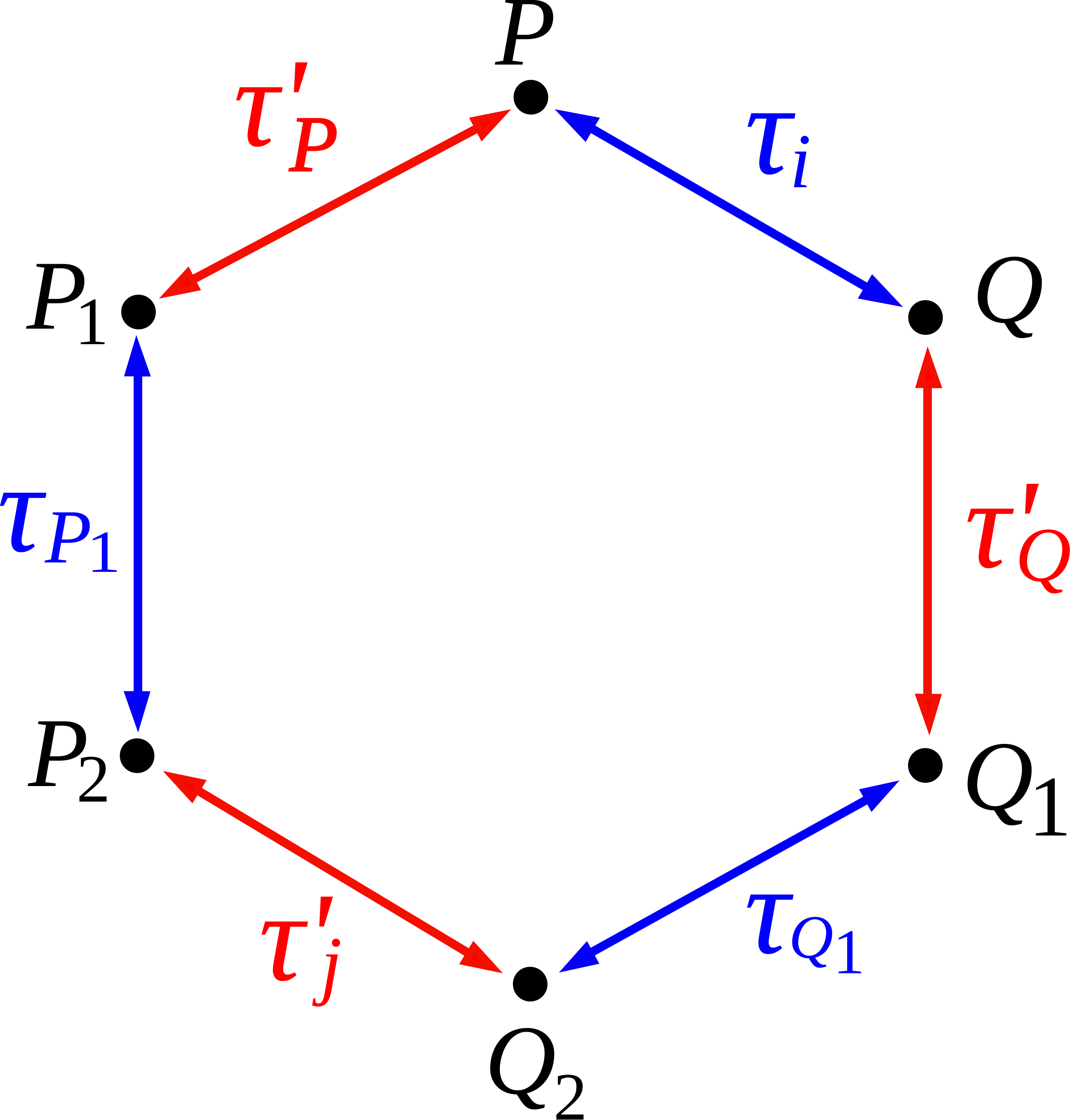}
    \caption{}
    \label{fig:hexagon}
\end{figure}
\end{Rmk}

\textbf{Partial translations.} For $g\in \calH_n$, it is useful to decompose $g$ into \emph{essential part} and \emph{finitary part}. Let $E(g)$ and $F(g)$ denote the restrictions of $g$ on $\Es g$ and $\supp g \setminus \Es g$ respectively. As an element of $\Sym_n$, the restriction of $g$ on $\Es g$ has a cycle decomposition. The set $\Es g$ is partitioned into orbits of essential points of $g$. The element $g$ restricts on each orbit to define an infinite cycle, which will be called a \emph{partial translation} of $g$. Denoting partial translations of $g$ by $\pt_{1} ,\cdots, \pt_{\ell}$, we can write $g$ as \begin{equation}\label{eq:partial_translation}
g=\pt_{1} \cdots \pt_{\ell}\cdot f
\end{equation} where $\supp f \subset \supp g \setminus \Es g$. Partial translations of $g$ commute with each other. In case $g$ has no essential points, (\ref{eq:partial_translation}) becomes $g= f$.
Note that a partial translation is not an element of $\calH_n$ in general. For example, $g_2^2$ has two partial translations, which are not eventual translations.

As an infinite cycle on $X_n$, each partial translation $\pt$ can be realized as an embedding $\bbz \hookrightarrow X_n$. Picking a base point $x_0\in \supp \pt$ we identify $\supp \pt =\{(x_0)\pt^k|k\in \bbz\}$ with $\bbz$. Let $[k]_\pt$ denote the point of $\supp \pt$ corresponding to $k\in \bbz$:
\begin{equation}\label{eq:identification_with_Z}
(x_0)\pt^k  \leftrightarrow [k]_\pt.
\end{equation}Under the identification (\ref{eq:identification_with_Z}), a partial translation $\pt$ of $g$ translates points on $\supp \pt$ by $+1$ in the above identification, i.e.,
\begin{equation}\label{eq:translation on pt}
[k]_\pt g^m =[k]_\pt \pt^m \leftrightarrow [k+m]_\pt
\end{equation}for each $m\in \bbz$.

By the definition (\ref{eq:eventual_translation}) an eventual translation $g \in\calH_n$ acts as a translation by $m_i$ on $R_i$ outside a finite set $F_g$ where $\pi(g) =(m_1, \cdots, m_n)$. We call a ray $R_i$ a \emph{source} of $g$ if $m_i<0$, and we call a ray $R_i$ a \emph{target} of $g$ if $m_i>0$. As a restriction of $g$ on one of its orbits, a partial translation $\pt$ of $g$ has unique \emph{source and target} so that $\pt$ moves points \emph{from the source towards the target}. More precisely there exist exactly two rays $R^-_\pt$ and $R^+_\pt$ in $\{R_1, \cdots, R_n\}$, and $k_0\in \bbn$ such that
\begin{equation}\label{eq:source_to_target}
[-k]_\pt \in R^-_\pt \text {   and   } [k]_\pt \in R^+_\pt
 \end{equation} for all $k\geq k_0$.

Note that the generators $g_2, \cdots, g_{n}$ of $\calH_n$ are all partial translations. Generators $g_i$'s share a unique source $R_1$ on which the actions of $g_i^{-1}$'s coincide as translations by $+1$. Moreover those generators have all distinct targets. Intuitively this behavior can be characterized as, for each pair $i\neq j$,
\begin{enumerate}\label{eq:cold_fact_for_a_couple}
\item[$(1)$] two actions of $g_i$ and $g_j$ are identical on $R_1\setminus (1,1)$,
\item[$(2)$] once $g_i$ and $g_j$ `diverge' from $(1,1)$ then they never meet again,
\item[$(3)$] the commutator $[g_i, g_j]=\alpha$ is the transposition which exchanges $(1,1)$ and $(1,2)$.
\end{enumerate}
Note that $\supp \alpha$ consists of two points which are the last two points before $g_i$ and $g_j$ diverge. It turns out that generators $\phi(g_2),\cdots, \phi(g_{n})$ of $\phi(\calH_n)$ follow the same rule provided $\phi$ is a monomorphism. See Lemmas~\ref{lemma:partial_translation} and \ref{lemma:common transp on pt}.

\begin{Example}\label{ex:mono}
Consider the monomorphism $\phi:\calH_3\to \calH_3, g_i\mapsto g_i^2$ for $i=2,3$. The generators $g_2^2$ and $g_3^2$ of $\phi(\calH_3)$ share a common source $R_1$ and their targets, $R_2$ and $R_3$ respectively, are distinct. Both $g_2^2$ and $g_3^2$ are products of two partial translations. Let $\pt_i$ and $\pt'_i$ denote the two partial transpositions of $g_i^2$ moving $(1,1)$ and $(1,2)$ respectively for $i=2,3$. Note that $R_1$ is the unique source for all four partial translations. The ray $R_2$ is the target of $\pt_2$ and $\pt'_2$, and $R_3$ is the target of $\pt_3$ and $\pt'_3$.  As in (\ref{eq:identification_with_Z}), we label $\supp \pt_2$ and $\supp \pt_3$ by $\bbz$ with a base point $(1,3)$:
$$
(1,3) \leftrightarrow [0]_{\pt_2}\text{  and  } (1,3) \leftrightarrow [0]_{\pt_3}.
$$ It is direct to check that $[k]_{\pt_2} = [k]_{\pt_3}$ for all $k\leq 0$ and so $\pt_2$ and $\pt_3$ agree on $\{[k]_{\pt_2} |k\leq 1\}$. Note that $[\pt_2, \pt_3] $ is the transposition which exchanges $(1,3) \leftrightarrow [0]_{\pt_2}$ and $(1,1) \leftrightarrow [1]_{\pt_2}$, the last two consecutive points before they diverge. The partial translations $\pt'_2$ and $\pt'_3$ follow the same rule: they agree on the source $R_1$ except $(1,2)$, and $[\pt'_2, \pt'_3]$ exchanges $(1,2)$ and $(1,4)$, the last two consecutive points before they diverge.

We also remark that two pairs of partial translations induce decomposition of the underlying set $X_3$. Let $X$ denote the orbit of $(1,1)$ under $\phi(\calH_3)$. In other words, $X$ is the union of two supports of $\pt_2$ and $\pt_3$. Similarly $\pt'_2$ and $\pt'_3$ determine a subset $X'$ which is the orbit of $(1,2)$ under $\phi(\calH_3)$. Obviously two invariant subsets provide a partition $X_3=X\sqcup X'$. The corresponding decomposition of $\phi(\calH_3)$ says that $\phi(\calH_3)$ is the disjoint union of two identical subgroups $H_3$ and $H'_3$ where $H_3$ and $H'_3$ are generated by $\pt_i$ and $\pt'_i$ respectively, $i=2,3$. The isomorphism between $H_3$ and $\calH_3$ can be described by the obvious bijection between $X$ and $X_3$. In particular, $\phi$ is injective.
\end{Example}

\begin{Lemma}\label{lemma:partial_translation}
A monomorphism $\phi:\calH_n \to \calH_n$ determines $\ell \in \bbn$ and a ray $R_l$
such that $R_l$  is a unique common source of $\phi(g_i)$,
and $\phi(g_i)$ has translation length $-\ell$ on $R_l$ for all $2\leq  i\leq n$.
Moreover targets of $\phi(g_i)$ and $\phi(g_j)$ do not share a ray if $2\leq i \neq j\leq n$.
\end{Lemma}

\begin{proof}
The group $\calH_2$ has only one generator $g_2$ of infinite order. So $\phi(g_2)$ satisfies the above automatically. To deal with a monomorphism $\phi$ of $\calH_n$ with $3\leq n$ recall notations:
$S_\alpha= \supp \phi(\alpha)$,
$T_\alpha=$ the set of all transpositions in (\ref{eq:cycle_alpha}),
$T_i\subset T_\alpha$ consisting of transpositions which intersect $\Es(\phi(g_i))$.
Being a monomorphism, $\phi$ implies that the order of every $\phi(g_i)$ is infinite.
So $\phi(g_i)$ contains at least one partial translation when written as in (\ref{eq:partial_translation}).
Let $\calP_i$ denote the set of all partial translations of $\phi(g_i)$.

Fix $i$. Lemma \ref{lemma:transposition_on_ess_orbit} says that $T_i\neq \emptyset$, and so $\calP_i$ contains a partial translation $\pt$ such that
\begin{equation}\label{eq:initial condition}
\supp \pt \cap S_\alpha\neq \emptyset.
\end{equation}
Consider the source $R_l$ of $\pt$ with the property (\ref{eq:source_to_target}).
Since $R_l$ is also a source of $\phi(g_i)$,
$\phi(g_i)$ has translation length $-\ell<0$ on $R_l$.

\textbf{Step 1.}
In this step we show $\phi(g_j)$ has the same translation length $-\ell<0$ on $R_l$ for all $j\neq i$.
The identities (\ref{eq:r_7}) imply that two involutions $\phi(\alpha)^{\phi(g_i)^{-k}}$ and $\phi(\alpha)^{\phi(g_j)^{-k}}$ are identical for all $k\in \bbn$.
So the two involutions must share their supports,
 \begin{equation}\label{eq:translstion_by_-k}
 (S_\alpha)\phi(g_i)^{-k}=(S_\alpha)\phi(g_j)^{-k}
 \end{equation} for all $k\in \bbn$.
Since $\pt$ intersects $S_\alpha$, one can take $A\in\supp\pt\cap S_\alpha\ne\emptyset$.
By \eqref{eq:source_to_target}, we have
$$
(A)\pt^{-k}\in\supp\pt\cap R_l
$$
for all $k\geq k_0$ where $k_0\in \bbn$.
Hence $(A)\pt^{-k}\in (S_\alpha)\pt^{-k}\cap R_l$.
Moreover, $(S_\alpha)\pt^{-k}\subset (S_\alpha)\phi(g_i)^{-k}$ for each integer $k$
because $\pt$ is the restriction of $\phi(g_i)$ on $\supp \pt$.
Therefore $(S_\alpha)\phi(g_i)^{-k}\cap R_l$ is non-empty for all but finitely many $k$.
The identities (\ref{eq:translstion_by_-k}) imply
\begin{equation}\label{eq:translstion_on_R_phi}
(S_\alpha)\phi(g_i)^{-k}\cap R_l=(S_\alpha)\phi(g_j)^{-k}\cap R_l
\end{equation}
for all $k\in \bbn$.
The identities (\ref{eq:translstion_on_R_phi}) force $\phi(g_j)$
to have the same translation length $-\ell<0$ on $R_l$.
Consider the two smallest balls on $X_n$ (centered at the origin) which contain two sets in (\ref{eq:translstion_on_R_phi}) respectively. If $\phi(g_j)$ moves points on $R_l$ (up to a finite set) by $\ell'\neq -\ell$, the two balls have to have different sizes for infinitely many $k$. This means $\phi(g_i)$ and $\phi(g_j)$ fail to satisfy the identities (\ref{eq:translstion_on_R_phi}). Therefore $\phi(g_i)$ and $\phi(g_j)$ have the same translation length $-\ell<0$ on $R_l$. Since $i$ and $j$ were arbitrary the first assertion of the Lemma is verified. The uniqueness of the ray will be verified in the last step.

\textbf{Step 2.}
Identify $\supp \pt$ with $\{[k]_\pt:k\in \bbz\}$ as described in (\ref{eq:identification_with_Z}).
We will show that there exists $k_0 \in \bbz$ such that for each $j\neq i$,
$\pt$ intersects only one partial translation $\qt \in \calP_j$ satisfying
\begin{equation}\label{eq:diverge}
\pt=\qt \text{  on  } \{[k]_{\pt}: k\leq k_0 \}
\end{equation} and
\begin{equation}\label{eq:common_support}
\supp \pt \cap \supp \qt = \{[k]_{\pt}: k \leq k_0+1 \}.
\end{equation}

From Step $1$ we see that actions of $\phi(g_j)$ and $\phi(g_i)$ are identical on $R_l$ as a translation by $-\ell<0$ up to a finite set. In particular, there exists $k_1 \in \bbz$ such that
\begin{equation}\label{eq:identical action of i and j}
(P)\phi(g_i) = (P)\phi(g_j)
\end{equation} for all $P \in\{[k]_\pt: k \leq k_1\}$. The orbit of (any such point) $P$ under $\phi(g_j)$ determines a partial translation $\qt$ as $\pt$ is the restriction of $\phi(g_i)$ on the orbit of $P$ under $\phi(g_i)$.
From the identity (\ref{eq:identical action of i and j}) it is immediate to check that $\qt$ satisfies
\begin{equation}\label{eq:map pt}
\qt=\pt \text{  on  } \{[k]_\pt: k \leq k_1\}.
\end{equation}
We first claim that $\pt \neq \qt$. If $\pt=\qt$ then $\phi(\alpha)$ restricted to $\supp \pt$ is the identity map since
$$(P)\phi(g_i) \phi(g_j) = (P)\pt \qt =(P) \pt^2 =(P)\qt\pt=(P)\phi(g_j) \phi(g_i)
$$ for all $P \in \supp \pt$. So $\pt=\qt$ implies that $\pt$ does not intersect $\phi(\alpha)$. However this contradicts the condition (\ref{eq:initial condition}).

From the condition (\ref{eq:map pt}) together with $\pt \neq \qt$ we can consider the largest $k_2\geq 0$ such that
$$
([k_1+k_2]_\pt)\pt = ([k_1+k_2]_\pt)\qt,
$$
or equivalently $k_2$ is the smallest number such that
$$
([k_1+k_2+1]_\pt)\pt \neq ([k_1+k_2+1]_\pt)\qt.
$$
It is obvious that $k_0= k_1+k_2$ satisfies (\ref{eq:diverge})
and $\supp \pt \cap \supp \qt \supset \{[k]_{\pt}: k \leq k_0+1 \}$.
For the reverse inclusion we assume $[k]_\pt \in  \supp \qt$ with $k\geq k_0+2$
and then draw a contradiction using the identities (\ref{eq:r_6}). Observe that $[k_0]_\pt \in S_\alpha$. By definition of $k_0$, $[k_0]_\pt$ satisfies
$$
 ([k_0]_\pt)\phi(g_i) \phi(g_j) = ([k_0]_\pt)\pt \phi(g_j) = ([k_0+1]_\pt) \phi(g_j) = ([k_0+1]_\pt) \qt
$$
which is distinct from
$$
 ([k_0+1]_\pt) \pt= ([k_0+1]_\pt) \phi(g_i)=([k_0]_\pt)\qt \phi(g_i)  = ([k_0]_\pt)\phi(g_j)\phi(g_i).
$$
So there exists a transposition $\tau \in T_\alpha$ which moves $[k_0]_\pt$, and hence
\begin{equation}\label{eq:first_pt}
 [k_0]_\pt \in S_\alpha.
\end{equation}
Moreover the condition (\ref{eq:diverge}) implies that $[k]_\pt \notin S_\alpha$ for all $k< k_0$, i.e., $k_0$ is the smallest number in
\begin{equation}\label{eq:pt and S}
N_\pt=\{k\in \bbz : [k]_\pt \in S_\alpha\}.
\end{equation}

The identities in (\ref{eq:r_6}) imply that $[\beta_k,\phi(g_j)]=1$ for all $k\geq k_0+2$ where $\beta_k$ is the conjugation defined by
\begin{equation}\label{eq:beta_translated}
\beta_k:=\phi(\alpha)^{\phi(g_i)^{k-k_0}}.
\end{equation}
To complete the proof we show that $\phi(g_j)$ does not commute with  $\beta_k$ under the assumption $[k]_\pt \in \supp \qt$ for some $k\geq k_0+2$. Our claim follows basically from that the involution $\beta_k\in \FSym_n$ intersects $\qt$. Since the cycle decomposition of $\beta_k$ contains $\tau_k:=\tau^{\phi(g_i)^{k-k_0}}$,
$\beta_k$ intersects $\qt$ at the point
$$
([k_0]_\pt)\phi(g_i)^{k-k_0} = ([k_0]_\pt)\pt^{k-k_0} = [k_0+(k-k_0)]_\pt = [k]_\pt
$$
by the observation (\ref{eq:translation on pt}).

If $\supp \tau_k \subset \supp \qt$ then we can apply a similar argument used in Corollary \ref{cor:normal subgroup} which shows that an infinite order element $g$ does not commute with a transposition $\beta$ if $\supp \beta\subset \Es g$. Identify $\supp \qt$ with $\bbz$ with an appropriate base point so that $[k]_\pt = [k]_\qt$. From $[k]_\qt \in \supp \tau_k$ and the minimality of $k_0 \in N_\pt$ we need to consider two cases: $\tau_k =([k]_\qt,\, [k+1]_\qt)$ or $\tau_k =([k]_\qt,\, [k']_\qt)$ with $k'\geq k+2$. For the first case we check that $([k]_\qt)\beta_k \phi(g_j) \neq ([k]_\qt)\phi(g_j)\beta_k$ by
$$
[k]_\qt\xrightarrow{\beta_k}([k]_\qt)\tau_k = [k+1]_\qt \xrightarrow{\phi(g_j)}([k+1]_\qt)\qt=[k+2]_\qt,
$$
$$
[k]_\qt\xrightarrow{\phi(g_j)} ([k]_\qt)\qt = [k+1]_\qt\xrightarrow{\beta_k}
([k+1]_\qt)\tau_k = [k]_\qt.
$$ So $ \supp [\beta_k, \phi(g_j)] \ni [k]_\qt $ is nontrivial, and hence $[\beta_k, \phi(g_j)] \neq 1$. In case  $\tau_k$ exchanges $[k]_\pt$ and $[k']_\qt$ with $k'\geq k+2$ we check $[\beta_k, \phi(g_j)]$ applied to $[k]_\qt$;
$$
[k]_\qt\xrightarrow{\beta_k}([k]_\qt)\tau_k = [k']_\qt \xrightarrow{\phi(g_j)}([k']_\qt)\qt=[k'+1]_\pt \xrightarrow{\beta_k^{-1}}([k'+1]_\qt)\beta_k^{-1} \xrightarrow{\phi(g_j)^{-1}}([k'+1]_\qt)\beta_k^{-1} \phi(g_j)^{-1}.
$$If $\beta_k$ commutes with $\phi(g_j)$ then we have to have $([k'+1]_\qt)\beta_k^{-1} \phi(g_j)^{-1} = [k]_\qt$ or equivalently
$$
([k'+1]_\qt)\beta_k =( [k]_\qt)\phi(g_j)=[k+1]_\qt.
$$ In other words $[\beta_k, \phi(g_j)] = 1$ implies that $\beta_k$ contains a transposition exchanging $[k+1]_\qt$ and $[k'+1]_\qt$. Inductively one applies the same argument to show that $\beta_k$ contains infinitely many transpositions in its cycle decomposition. Consequently $\supp \beta_k$ is an infinite set provided $[\beta_k, \phi(g_j)] = 1$. However every element of $\FSym_n$ must have a finite support.

We can apply an analogous argument when $\tau_k$ exchanges $[k]_\qt$ and $P \notin \supp \qt$. Our claim is that $\beta_k$ has an infinite support if it commutes with $\phi(g_j)$. From $\beta_k \phi(g_j) = \phi(g_j) \beta_k$ and
$$
[k]_\qt\xrightarrow{\beta_k}([k]_\qt)\tau_k = P \xrightarrow{\phi(g_j)}(P)\phi(g_j),
$$
$$[k]_\pt \xrightarrow{\phi(g_j)}([k+1]_\qt) \xrightarrow{\beta_k}([k+1]_\qt)\beta_k
$$we see that $(P)\phi(g_j) = ([k+1]_\qt)\beta_k$. Since $\supp \qt$ is an $\phi(g_j)$-invariant subset of $\supp \phi(g_j)$, $(P)\phi(g_j) \notin \supp \qt$. In particular $[k+1]_\qt \in \supp \beta_k$. By induction argument as before, we can show that $\supp \beta_k$ is an infinite set, which is a contradiction.

So far we have found a partial translation $\qt \in \calP_j$ and $k_0 \in \bbz$ with the properties (\ref{eq:diverge}) and (\ref{eq:common_support}). For the uniqueness of $\qt$ suppose $\pt$ intersects $\qt' \in \calP_j$ with $\qt' \neq \qt$. Since $\qt$ does not intersect $\qt'$, if $[k]_\pt \in \supp \qt'$ then $k\geq k_0+2$ by (\ref{eq:diverge}). One can apply the same argument as above to show that $\beta_k$ defined in (\ref{eq:beta_translated}) does not commute with $\qt'$, which contradicts (\ref{eq:r_6}).

We remark that the smallest number $k_0 \in N_\pt$ does not depend on $j$. This follows from the relation $\alpha=[g_i,g_j]$ of (\ref{eq:relations}) which implies that $\phi(\alpha)=\phi([g_i, g_{j'}])=\phi([g_i, g_{j}])$ for pairwise distinct $i,j$ and $j'$ in $\{2,\cdots, n\}$.
The minimality of $k_0$ in $N_\pt$ implies that $\phi(g_{j'})$ determines a unique partial translation $\qt_{j'}\in \calP_{j'}$ which agrees with $\pt$ on $\{[k]_\pt:k\leq k_0\}$ and that $([k_0]_\pt)\qt_{j'}^2 \neq ([k_0]_\pt)\pt^2$. In other words $k_0$ satisfies identities (\ref{eq:diverge}) and (\ref{eq:common_support}) where $\qt$ is replaced by $\qt_{j'}$.

\textbf{Step 3.} In this step we verify the second part of the Lemma. For $2\leq i\leq n$, let $\frakT_i$ denote the set of all targets of $\phi(g_i)$, i.e., $\phi(g_i)$ has a positive translation length on a ray $R_k \in \frakT_i$. If $\frakT_i$ and $\frakT_j$ share a ray $R_k$ then $\phi(g_i)$ and $\phi(g_j)$ have positive translation lengths on $R_k$. Using the (positive) least common multiple of the two translation lengths one can find $\pt\in \calP_i$ and $\qt \in \calP_j$ such that $\supp\pt \cap \supp \qt$
contains infinitely many points on the target $R_k$. However this pair of partial translations fail to satisfy the conditions (\ref{eq:diverge}) and (\ref{eq:common_support}) of Step $2$. Therefore $\phi(g_i)$ and $\phi(g_j)$ can not share any ray in their targets provided $i\neq j$.

\textbf{Step 4.} From Step $1$, we know that there exists a ray $R_l$ which is a source of $\phi(g_i)$ for all $2\leq i\leq n$. So each target $\frakT_i$ is a subset of $\{R_1, \cdots, R_n\} \setminus \{R_l\}$. The only way for the targets $\frakT_2, \cdots, \frakT_{n}$ to be pairwise disjoint is when they are all distinct singleton sets. A monomorphism $\phi$ defines a bijection $\gamma:\{2,\cdots, n\} \to \{1, \cdots, \hat{l}, \cdots, n\}$ such that
\begin{equation}\label{eq:permutation on n-1}
\frakT_i = \{R_{\gamma(i)}\}
\end{equation} for $2\leq i\leq n$. Now it is obvious that $R_l$ is only ray on which $\phi(g_i)$ has a negative translation length for all $i$. Consequently $\phi$ determines a unique ray $R_l$ which is a common source of $\phi(g_i)$ for all $i$.
\end{proof}

As we checked in Step $2$ in the proof of Lemma \ref{lemma:partial_translation}, if two partial translations $\pt\in \calP_i$ and $\qt \in \calP_j$ ($i \neq j$) intersect then there exists a unique point $[k_0]_\pt$ satisfying conditions (\ref{eq:diverge}) and (\ref{eq:common_support}). Intuitively those two conditions can be interpreted as (1) for each $\pt \in \calP_i$ there exists $\qt\in \calP_j$ such that $\pt$ and $\qt$ are identical all the way up to $[k_0]_\pt$, (2) once $\pt$ develops a different orbit from $[k_0+1]_\pt$, $\pt$ never intersects $\qt$ again. Compare the above with the characterization $(1)$ and $(2)$ of the generators of $\calH_n$ in page~\pageref{eq:cold_fact_for_a_couple}. With this intuition, let us call the point $[k_0+1]_\pt$ \emph{the diverging point} of $\pt$ and denote it by $D_\pt$. For all partial translations in $\calP=\cup _{i=2}^{n}\calP_i$ we label those supports with $\bbz$ appropriately such that
\begin{equation}\label{eq:convention}
D_\pt =[0]_\pt
\end{equation}for each $\pt\in \calP$. From now on let us use the above default labeling for $\supp \pt$ for all $\pt \in \calP$.


\begin{Lemma}\label{lemma:common transp on pt}Let $i \in \{ 2, \cdots, n\}$.
Each partial translation $\pt$ of $\phi(g_i)$ intersects exactly one transposition $\tau_\pt\in T_i$ which exchanges $[-1]_\pt$ and $[0]_\pt$.
\end{Lemma}
\begin{proof}
Any transposition $\tau \in T_i$ is one of the three types:
\begin{enumerate}
\item [] Type I: $\tau$ exchanges two consecutive points of one partial transposition of $\calP_i$.
\item [] Type II: $\tau$ exchanges two non-consecutive points of one partial transposition of $\calP_i$.
\item [] Type III: $\tau$ intersects two distinct partial transpositions of $\calP_i$.
\end{enumerate}We show that each $\pt \in \calP_i$ can intersect only one transposition $\tau \in T_i$ of type I. From (\ref{eq:first_pt}) we see that $k_0 =-1$ is indeed the smallest integer in $N_\pt$ defined in (\ref{eq:pt and S}).  We first focus on the case $n\geq3$.

\textbf{Claim I:} If $\pt \in \calP_i$ intersects $\tau$ of type I, then $\pt$ does not intersect any other transposition of $T_i$. The first case to consider is when $\tau$ exchanges $[-1]_\pt$ and $[0]_\pt$ where $-1$ is the smallest number in $N_\pt$. Suppose $\pt$ intersects $\tau' \in T_\alpha$ of any type whose support contains $[q]_\pt$. Since $\tau$ and $\tau'$ does not intersect, $q\geq 1$.

In case $q=1$, we use the relation $r_2$ of (\ref{eq:relations}) and the hexagon argument described in Remark \ref{rmk:hexagon}. Two involutions $\phi(\alpha)$ and $\beta:=\phi(\alpha)^{\phi(g_i)}$ satisfy
$$
(\phi(\alpha) \beta)^3 =1.
$$ Observe that $\beta$ contains a transposition $\tau^{\phi(g_i)} =\tau^{\pt} = ([0]_\pt, [1]_\pt)$ which intersects two transpositions $\tau$ and $\tau'$ of $\phi(\alpha)$. Therefore there exists a hexagon on $6$ points
$$
[-1]_\pt\xrightarrow{\tau}[0]_\pt\xrightarrow{\tau^{\pt}}[1]_\pt
\xrightarrow{\tau'}Q\xrightarrow{\beta}
(Q)\beta\xrightarrow{\phi(\alpha)}S\xrightarrow
{\beta}[-1]_\pt
$$ where $Q=([1]_\pt)\tau'$ and $S=([-1]_\pt)\beta$. In particular, $\supp \beta$ contains $[-1]_\pt$. This means that $S_\alpha$ contains $[-2]_\pt$, a contradiction to the minimality of $-1 \in N_\pt$.

Next, we may assume $q \geq 2$ and $q$ is the smallest number in $N_\pt\setminus \{-1, 0\}$. Consider the conjugation defined by
\begin{equation*}
\beta=\phi(\alpha)^{\phi(g_i)^{q}}.
\end{equation*} Since $q\geq2$ the identity (\ref{eq:r_8}) says that $\phi(\alpha)$ and $\beta$ commute. Two involutions $\phi(\alpha)$ and $\beta$ move $[q]_\pt$ to distinct points since
\begin{equation*}\label{eq:distinct points}
[q]_\pt \phi(\alpha)= ([q]_\pt) \tau' \neq [q-1]_\pt =([q]_\pt)\beta =([q]_\pt)\tau^{\pt^{q}}=[q-1]_\pt
\end{equation*} where the inequality follows from the minimality of $q$ in $N_\pt\setminus \{-1, 0\}$.
We are in good position to apply the `square argument' described in Remark \ref{rmk:square_in_commuting_involutions}. The square on $I=S_\alpha \cap \supp \beta$ involving $4$ points
\begin{equation*}
[q-1]_\pt\xrightarrow{\beta}[q]_\pt\xrightarrow{\tau'}([q]_\pt)\tau'\xrightarrow
{\beta}([q-1]_\pt)\phi(\alpha) \xrightarrow{\phi(\alpha)}[q-1]_\pt
\end{equation*} implies that $[q-1]_\pt$ must belong to $N_\pt\setminus \{-1, 0\}$. Again, this contradicts to the minimality of $q$.

\textbf{Claim II:} If $\pt \in \calP_i$ intersects $\tau \in T_i$ then $\tau$ can not be of type II. First we argue that $\supp \tau$ dose not contain $[-1]_\pt$ if $\tau$ is of type II. Assume $\tau= ([-1]_\pt, [q]_\pt)$ for some $q\geq 1$. The involution $\beta$ defined by
\begin{equation*}
\beta=\phi(\alpha)^{\phi(g_i)^{q+1}}
\end{equation*} contains $\tau^{\pt^{q+1}}$ in its cycle decomposition. So $\beta$ moves $[q]_\pt$ to $[q+(q+1)]_\pt $ but $\phi(\alpha)$ moves $[q]_\pt$ to $[q-(q+1)]_\pt = [-1]_\pt$. Therefore we can apply the square argument to see that $[-1]_\pt$ belongs to $\supp \beta$, or equivalently $[-1-(q+1)]_\pt \in N_\pt$ as in the proof of Claim I. However, the minimality of $-1$ in $N_\pt$ says that this can not happen.

To complete the proof, we need to rule out the case when $\pt$ intersects $\tau = ([q]_\pt, [r]_\pt)$ for some $q\geq 0$ and $r\geq q+2$. Recall that $-1$ is the smallest number in $N_\pt$, and so $T_i$ contains a transposition $\tau'$ which moves $[-1]_\pt$. From the discussion so far we may assume that $\tau'$ is of type III. Consequently the point $[-1]_\pt$ gets mapped to distinct points by $\phi(\alpha)$ and $\beta$ defined by
\begin{equation*}
\beta=\phi(\alpha)^{\phi(g_i)^{-(r+1)}}
\end{equation*}since $([-1]_\pt)\beta =[-1-(r-q)]_\pt\in \supp \pt$ but $([-1]\pt)\phi(\alpha) = ([-1]_\pt)\tau' \notin \supp \pt$. The square argument applied to two involutions $\phi(\alpha)$ and $\beta$ implies that the point $[-1-(r-q)]_\pt$ belongs to $N_\pt$, which can not happen.

\textbf{Claim III:} If $\pt$ intersects $\tau \in T_i$ of type III then $\tau$ fixes the point $[-1]_\pt$ where $-1$ is the smallest number in $N_\pt$. Observe that Claim III together with Claim II implies that the transposition $\tau_\pt$ of $T_i$, which moves $[-1]_\pt$, must be of type I. Consequently $\pt$ intersects only transposition $\tau_\pt$ which exchanges $[-1]_\pt$ and $[0]_\pt$, completing the proof for the Lemma.

To this end we draw a contradiction from the assumption that $\pt$ intersects $\tau \in T_i$ of type III with $[-1]_\pt \in \supp\tau$. Since $\tau$ is of type III, it relates $\pt$ to a partial translation $\pt_1 \in \calP_i$ that $\tau$ intersects. From Step II of Lemma \ref{lemma:partial_translation} we know that there exists a smallest number in
$$
N_{\pt_1} = \{k\in \bbz : [k]_{\pt_1} \in S_\alpha\}.
$$ By the convention \ref{eq:convention} we can further say that the smallest number is $-1$ (after pre-composing an appropriate translation on $\bbz$ to $\bbz \hookrightarrow \supp \pt_1$ if necessary). We can show that if $([-1]_\pt)\tau = [k_1]_{\pt_1}$ then $k_1\geq 0$. If $k_1 = -1$, we have to have
$$
([-1]_\pt)\pt \qt = ([-1]_\pt)\phi(g_i) \phi(g_j)  = ([-1]_{\pt_1})\phi(g_j) \phi(g_i) =([-1]_{\pt_1}) \qt_1 \pt_1
$$ where $\qt$ and $\qt_1$ are unique partial translations of $\calP_j$ which intersect $\pt$ and $\pt_1$ respectively. Consequently $\qt$ intersects both of $\pt$ and $\pt_1$, a contradiction to the uniqueness of $\qt$. Therefore $\tau$ fixes $[-1]_{\pt_1}$ and so $k_1 \geq 0$.

The partial translation $\pt_1$ intersects $\tau$ of type III and so Claim I implies that $\pt_1$ can only intersect transpositions of type III. In particular the transposition $\tau_1 \in T_i$ with $[-1]_{\pt_1} \in \supp \tau_1$ must be of type III. Observe that $\tau_1$ relates $\pt_1$ to $\pt_2 \in \calP_i \setminus \{\pt_1\}$ as $\tau$ relates $\pt$ to $\pt_1$. As argued above one can check that $\tau_1$ intersects $\pt_2$ at a point $[k_2]_{\pt_2}$ with $k_2>-1$ where $-1$ is the smallest number in $N_{\pt_2}= \{k\in \bbz : [k]_{\pt_2} \in S_\alpha\}$. Inductively we can consider a cycle of partial translations
\begin{equation}\label{eq:cycle pt}
\pt_0 = \pt \xrightarrow{\tau_0= \tau}\pt_1\xrightarrow{\tau_1} \cdots \xrightarrow
{\tau_{m-1}}\pt_{m}\xrightarrow
{\tau_{m}}\pt_0.
\end{equation} Let $k_1, \cdots, k_{m}$ be integers so that $\tau_r = ([-1]_{\pt_r}, [k_{r+1}]_{\pt_{r+1}})$ for $0\leq r\leq m-1$ and $\tau_m = ([-1]_{\pt_m}, [k_m]_{\pt_0})$.

Next we show that $k_r = 0$ for $1\leq r\leq m$. Assume the contrary that, for example, $k_1 \geq 1$. We can apply the square argument to two commuting involutions $\phi(\alpha)$ and $\beta$ defined by
\begin{equation*}
\beta=\phi(\alpha)^{\phi(g_i)^{k_1+1}}
\end{equation*} since $([k_1]_{\pt_1})\phi(\alpha) = [-1]_{\pt_0} \neq [k_2+k_1+1]_{\pt_2} = ([k_1]_{\pt_1})\beta$. Consequently $[-1]_{\pt_0} \in \supp \beta$ or equivalently $-1-(k_1+1) \in N_{\pt_0}$, contradicting to the minimality of $-1 \in N_{\pt_0}$. A similar argument shows that $k_r = 0$ for all $r$.

So far we have shown that if $\pt$ intersects $\tau\in T_i$ of type III then there exists a cycle of partial translations (\ref{eq:cycle pt}) which are related by transpositions $\tau_r = ([-1]_{\pt_r}, [0]_{\pt_{r+1}})$ for $0\leq r\leq m-1$ and $\tau_m = ([-1]_{\pt_m}, [0]_{\pt_0})$. Next we show that $\phi(g_i)$ restricts on the set $C_1=\{[1]_{\qt_0}, \cdots, [1]_{\qt_m}\}$ to define a $(m+1)$-cycle
$$
\sigma_1:[1]_{\qt_m} \to [1]_{\qt_{m-1}} \to\cdots \to [1]_{\qt_0}\to [1]_{\qt_m}.
$$ For each $0\leq r\leq m-1$, the transposition $\tau_r = ([-1]_{\pt_r}, [0]_{\pt_{r+1}})$ of $[\phi(g_i), \phi(g_j)]$ implies that
$$
[-1]_{\pt_r} \xrightarrow{\phi(g_i)}[0]_{\pt_r}= [0]_{\qt_r}\xrightarrow{\phi(g_j)} [1]_{\qt_r} \xrightarrow
{\phi(g_i)^{-1}}([1]_{\qt_r}){\phi(g_i)^{-1}}\xrightarrow
{{\phi(g_j)^{-1}}}[0]_{\qt_{r+1}}.
$$
From the last arrow we have $([1]_{\qt_r}){\phi(g_i)^{-1}}=[1]_{\qt_{r+1}}$, or $([1]_{\qt_{r+1}})\phi(g_i)= [1]_{\qt_r}$ for $0\leq r\leq m-1$. For $r=m$ it is direct to check that $[1]_{\qt_0}\phi(g_i) = [1]_{\qt_m}$. Therefore $\phi(g_i)$ restricted on the set $C_1$ is $\sigma_1$.

Finally we show that $\phi(g_i)$ contains infinitely many copies of $\sigma_1$ in its cycle decomposition. Observe that the conjugation $\phi(\alpha)^{\phi(g_j)^2}$ contains a factor $\beta_2$ given by
\begin{equation}\label{eq:beta_2}
\beta_2= (\tau_0 \tau_1 \cdots \tau_m)^{\phi(g_j)^2}.
\end{equation} The action of ${\phi(g_j)^2}$ on $\supp \beta$ is the translation by $+2$ along partial translations $\qt_0, \qt_1, \cdots, \qt_m$. So $\beta_2$ can be written as a product of transpositions
$$
([1]_{\qt_0}, [2]_{\qt_1}) \cdots ([1]_{\qt_{m-1}}, [2]_{\qt_m})([1]_{\qt_{m}}, [2]_{\qt_0})
$$ Therefore the conjugation $(\sigma_1)^{\beta_2}$ is the $(m+1)$-cycle
$$
\sigma_2:[2]_{\qt_m} \to [2]_{\qt_{m-1}} \to\cdots \to [2]_{\qt_0}\to [2]_{\qt_m}.
$$ Intuitively speaking, taking the conjugation of $\sigma_1$ by $\beta_2$ results in translating the $(m+1)$-cycle $\sigma_1$ globally by $+1$ along the partial translations $\qt_0, \cdots, \qt_m$. Now the identity (\ref{eq:r_7}) says that $\phi(g_i)$ commutes with $\phi(\alpha)^{\phi(g_j)^2}$. In particular $\phi(g_i)^{\beta_2} = \phi(g_i)$. This means that $\phi(g_i)$ already contains $\sigma_2$, a translated copy of $\sigma_1$, in its cycle decomposition. Applying identity (\ref{eq:r_7}) repeatedly to $\phi(g_i)$ and $\phi(\alpha)^{\phi(g_j)^k}$, we see that $\phi(g_i)^{\beta_k} = \phi(g_i)$ for $k\geq 2$ where $\beta_k$ is defined in an analogous way as in (\ref{eq:beta_2}) with the power $k$ instead of $2$. Thus $\phi(g_i)$ contains infinitely many copies of $\sigma_2$ in its cycle decomposition, which is absurd. In all, a partial translation $\pt$ which intersects $\tau \in T_i$ of type III forces that $\phi(g_i) \notin \calH_n$. Claim III is verified and so we are done.

If $n=2$, we can show the same result for a partial translation of $\phi(g_2)$ as a special case of the above discussion. The braid relation and commutation relation of (\ref{eq:presentation H_2}) can be used for Claim $1$ and Claim $2$ respectively. One can verify Claim $3$ immediately with braid relation again.
\end{proof}

\begin{Cor}\label{coro:summary} Let $3\leq n$. Suppose that a partial translation $\pt$ of $\phi(g_i)$ intersects a partial translation $\qt$ of $\phi(g_j)$. Then $\pt$ does not intersect any other cycles of $\phi(g_j)$ but $\qt$.
\end{Cor}
\begin{proof}
We first check that $[1]_\pt$ is fixed under $\phi(g_j)$. Lemma \ref{lemma:common transp on pt} states that the unique transposition $\tau_\pt$, which intersects $\pt$, exchanges $[-1]_\pt$ and $[0]_\pt$. From the identity $ ([0]_\pt)[\phi(g_i), \phi(g_j)] =([0]_\pt)\tau_\pt = [-1]_\pt$ we have $([0]_\pt)\phi(g_i) \phi(g_j) = ([-1]_\pt)\phi(g_j)\phi(g_i)$, and so
\begin{align*}
([1]_{\pt})\phi(g_j)& = ([0]_\pt)\pt \phi(g_j)= ([0]_\pt)\phi(g_i)\phi(g_j) = ([-1]_\pt)\phi(g_j)\phi(g_i) \\
&= ([-1]_\pt)\qt \phi(g_i) = ([0]_\pt)\phi(g_i) = ([0]_\pt) \pt = [1]_\pt
\end{align*}since $([-1]_\pt)\qt = ([-1]_\pt)\pt=[0]_\pt$ by (\ref{eq:diverge}). Suppose $\pt$ intersects a component $f \neq \qt$ in the cycle decomposition of $\phi(g_j)$. We may further assume that if $k_1$ is the smallest integer in $\{k\in \bbz: [k]_\pt \in \supp f\}$ then $[k]_\pt$ is fixed under $\phi(g_j)$ for all $1\leq k\leq k_1-1$. To show $ [k_1-1]_\pt \in \supp \phi(\alpha)$ we check that
$$
[k_1-1]_\pt \xrightarrow{\phi(g_i)} ([k_1-1]_\pt)\pt = [k_1]_\pt \xrightarrow{\phi(g_j)} ([k_1]_\pt)\phi(g_j) = ([k_1]_\pt)f
$$which is distinct from
$$
[k_1-1]_\pt \xrightarrow{\phi(g_j)} ([k_1-1]_\pt)\phi(g_j) = [k_1-1]_\pt \xrightarrow{\phi(g_i)} ([k_1-1]_\pt)\pt = [k_1]_\pt.
$$Since $k_1-1 > 0$, $\phi(\alpha)$ contains a transposition $\beta \neq \tau_\pt$ which intersects $\pt$. However this contradicts Lemma \ref{lemma:common transp on pt}.
\end{proof}

The unique positive integer $\ell$ in Lemma \ref{lemma:partial_translation} is called \emph{the eventual length} of a monomorphism $\phi$, and denoted by $\ell(\phi)$. Lemma \ref{lemma:partial_translation} implies that $\calP_i$ contains precisely $\ell=\ell(\phi)$ partial translations $\pt_1, \cdots, \pt_\ell$ for each $2\leq i\le n$. Since they are only infinite cycles of $\phi(g_i)$ we have $\Es \phi(g_i) = \cup_{j=1}^\ell \supp \pt_j$.
Lemma \ref{lemma:partial_translation} also implies that each ray of $X_n$ is either a source or a target of $\phi(g_i)$ for some $i$. So the set $\cup_{i=2}^n \Es \phi(g_i)$ contains all points of $X_n$ but finitely many. Let $\ES\subset X_n$ be the set
$$
\ES= \bigcup_{g\in \calH_n} \Es \phi(g).
$$

\begin{Prop}\label{prop:decomposition}
With the notation defined above, $\ES= \cup_{i=2}^n \Es \phi(g_i)$.
\end{Prop}
\begin{proof}
One side inclusion is clear. For the reverse inclusion we claim that
$F = X_n\setminus\cup_{i=2}^n \Es \phi(g_i)$ is invariant under $\phi(g)$ for all $g\in \calH_n$. Note that $P\in F$ if and only if $P$ has a finite orbit under $\phi(g_i)$ for all $2\leq i\leq n$. In other words each $\phi(g_i)$ contains a finite cycle which moves $P$. We argue by induction on the word length of $g$. If a finite cycle $f$ of $\phi(g_i)$ moves $P$ then obviously $(P)\phi(g_i) = (P)f$ has the same finite orbit under $\phi(g_i)$. If $(P)\phi(g_i)$ is fixed under $\phi(g_j)$ for some $j$ then we have nothing to prove because the orbit under $\phi(g_j)$ is trivial. To show $(P)\phi(g_i) \in F$ for every $i$, we need to check that if $\phi(g_i)$ has a finite cycle $f$ such that $(P)f \in \supp \phi(g_j)$ for some $j \neq i$, then $\phi(g_j)$ also has a finite cycle which moves $(P)\phi(g_i)$. This follows from the last assertion of Corollary \ref{coro:summary} which implies that if $f$ intersects $\phi(g_j)$ then it can only intersect finite cycles of $\phi(g_j)$ for all $j\neq i$. The same argument shows that if $(P)\phi(g_i)^{-1}$ belongs to $\supp \phi(g_j)$ then it has a finite orbit under $\phi(g_j)$ for every $j\neq i$. So $(P)\phi(g_i)^{\pm1} \in F$. This establishes the base case.

Suppose $g\in \calH_n$ with $|g|=k+1$. Say the last letter of $g$ is $g_i^{\pm1}$, i.e., $g= g'\cdot g_i^{\pm1}$ with $|g'|=k$. By induction assumption $(P)\phi(g') \in F$. So, for each $i$, if $(P)\phi(g') \in \supp \phi(g_i)$ then there exists a finite cycle of $\phi(g_i)$ which moves the point $(P)\phi(g')$. A similar argument applied to $(P) \phi(g')\phi(g_i)^{\pm1}$, instead of $(P)\phi(g_i)^{\pm1}$ as in the base case, shows that the point $(P)\phi(g')\phi(g_i)^{\pm1}$ has a finite orbit under $\phi(g_j)$ for all $2\leq j\leq n$. Therefore $(P)\phi(g)\phi(g_i)^{\pm1} \in F$, and hence $F$ is an invariant set under $\phi(g)$ for all $g\in \calH_n$. Since $F$ is finite $P \in F$ must have a finite orbit under $\phi(g)$. This means that $P \notin \ES$.
\end{proof}

Consider the decomposition of $X_n =\ES \sqcup F$ where $F$ is the finite set as in the above proof. Let $\E(\phi(g))$ denote the restriction of $\phi(g)$ on the set $\ES$. From the decomposition we have $\E(\phi(g)) \E(\phi(h)) = \E(\phi(gh))$ for all $g,h\in \calH_n$. One crucial observation is that, for each $2\leq i\leq n$, $\E(\phi(g_i)) $ is nothing but the product of its commuting partial translations;
\begin{equation}\label{eq:prod partial translations}
\E(\phi(g_i)) =\prod_{\pt \in \calP_i} \pt.
\end{equation} Consequently we have
\begin{equation}\label{eq:essup partial translations}
\Es (\phi(g_i)) = \bigcup_{\pt\in \calP_i} \supp \pt.
\end{equation}
So Lemma \ref{lemma:common transp on pt} implies that $T_i$ consists of $\ell$ transpositions each of which exchanges $[-1]_{\pt}$ and $[0]_{\pt}$ for some $\pt \in \calP_i$. Two conditions (\ref{eq:diverge}) and (\ref{eq:common_support}) show that $T_i = T_j$ for each pair $2\le i\neq j\le n$. It follows from that $\E(\phi(\alpha))$ is the product of $\ell$ transpositions in $T_i$ for any $i$. Indeed we have the following.

\begin{Cor}\label{cor:number of diverging pt}
Let $D$ be the set of all diverging points of partial translations in $\calP = \cup_{i=2}^n \calP_i$. There exists bijections
$$
\calP_i \leftrightarrow D \leftrightarrow T_i
$$for each $2\leq i \leq n$.
\end{Cor}
\begin{proof}
Fix $i$ and consider the sequence of maps
$$
\calP_i \xrightarrow{\iota_1} D \xrightarrow{\iota_2} T_i \xrightarrow{\iota_3}\calP_i
$$ defined by
$$
\iota_1:\pt \mapsto[0]_\pt , \; \iota_2:[0]_\pt \mapsto \tau_\pt, \; \iota_3:\tau_\pt \mapsto \pt.
$$ We claim that all three maps are injective whose composition in a row yields the identity. Partial translations of $\calP_i$ do not intersect each other, which explains $\iota_1$ is injective. Every diverging point is given by $[0]_\qt$ for some $\qt\in \calP_j$. If $j=i$, we have nothing to show that $\iota_2$ is well-defined due to Lemma~\ref{lemma:common transp on pt}. If not, from Step 2 of Lemma~\ref{lemma:partial_translation} one can find a partial translation $\pt \in \calP_i$ with conditions (\ref{eq:diverge}) and (\ref{eq:common_support}). So we have $[0]_\qt= [0]_\pt$. The assignment $[0]_\pt\mapsto \tau_\pt$ is well-defined and injective due to Lemma~\ref{lemma:common transp on pt}. For $\iota_3$ it suffices to check if $\tau \in T_i$ then $\tau$ intersects a unique partial translation in $\calP_i$. By definition, if $\tau \in T_i$ then $\tau$ intersects $\Es \phi(g_i)$. The union (\ref{eq:essup partial translations}) forces $\tau$ to intersect at least one partial translation of $\calP_i$. By Lemma~\ref{lemma:common transp on pt} again, $\iota_3$ is well-defined and injective. It follows immediately from definitions of three maps that the composition is the identity map. We have established bijections as desired.
\end{proof}


\textbf{Expanding map} $\tilde{\phi}$. Recall that $\calH_n$ acts on the underlying set $X_n=\{1, \cdots, n\}\times \bbn$ transitively. So $X_n$ can be considered as a single orbit of a base point $(1,1)$. Let us to express points $(j,m)\in X_n$ ($1\leq j\leq n$, $m\in \bbn$) as
\begin{equation*}\label{eq:transitive action}
 (j,m)=\begin{cases}
(1,1)g_2^{-(m-1)} &\text{  if  $j=1$},\\
(1,1)g_j^{m} &\text{  if  } 2\leq j\leq n.
\end{cases}
\end{equation*}
In case $(j,m)\in R_1$ we simply choose $g_2^{-1}$ among inverses of $n-1$ generators whose actions coincide on $R_1$ as translations by $+1$. By taking $g_2^{-1}$ as a default translation on $R_1$ we can make the above expression unique. Take subscripts and exponents of the generators to obtain the following coordinate system for $(j,m)\in X_n$
 \begin{equation}\label{eq:coordinate}
 (j,m)\leftrightarrow\begin{cases}
[2,-(m-1)] &\text{  if  $j=1$},\\
[j,m] &\text{  if  }  2\leq j\leq n.
\end{cases}
\end{equation}

Corollary~\ref{cor:number of diverging pt} implies that a monomorphism $\phi$ determines $\ell=\ell(\phi)$ diverging points $D_1, \cdots, D_\ell$. Let $\calQ_l =\{\pt\in \calP:[0]_\pt=D_l\}$ for $1\leq l\leq \ell$. For each $2\leq i\leq n$, $\calP_i$ contains precisely one partial translation $\pt$ such that $[0]_\pt=D_l$. Let $\pt_{l,i}$ denote such a unique partial translation $\pt \in \calP_i$, i.e., $\pt_{l,i} =\iota_1^{-1}(D_l)$ where $\iota_1:\calP_i \to D$ is the bijection defined in proof of Corollary~\ref{cor:number of diverging pt}. We label partial translations of $\calQ_l$ as
$$
\calQ_l=\{\pt_{l,2}, \cdots, \pt_{l,n}\}
$$With this new labeling, $\calP_i =\{\pt_{l,i}: 1\leq l \leq \ell\}$ for each $2\leq i\leq n$, and the set of all partial translations $\calP$ determined by $\phi$ has decompositions $\calP = \sqcup_{i=2}^n \calP_i=\sqcup_{l=1}^\ell \calQ_l$. The last identity follows from the equivalence relation: partial translations $\pt, \qt\in \calP$ belong to $\calQ_l$ for some $l$ if and only if $\pt$ intersects $\qt$. For better notation, let $[m]_{l,i}$ denote the point $[m]_{\pt_{l,i}}$ from now on. Let $\calO(D_l)\subset \ES$ denote the set
$$\calO(D_l)=\bigcup_{2\leq i\leq n} \supp \pt_{l,i}.
$$
Since $D_l =[0]_{l,i}$ for all $2\leq i\leq n$, every point in $\calO(D_l)$ can be written as $[m]_{l,i}=(D_l)\pt_{l,i}^m$ for some $\pt_{l,i}\in \calQ_l$ and $m\in \bbz$. Note that this expression is not unique when $m\leq 0$. Each pair of partial translations of $\calQ_l$ satisfies conditions (\ref{eq:diverge}) and (\ref{eq:common_support}). The first condition implies that the point $[m]_{l,i}$ with $m\leq 0$ and $3\leq i\leq n$ is identified to $[m]_{l,2}$. The second condition implies that $[m]_{l,i} = [m]_{l,i'}$ if and only if $i=i'$ when $m\geq 1$. So every point $P\in\calO(D_l)$ can be expressed uniquely as
 \begin{equation}\label{eq:coordinate_Q_l}
P=\begin{cases}
[m]_{l,2} &  m\leq0,\\
[m]_{l,i} &   2\leq i\leq n, 1\leq m.
\end{cases}
\end{equation}
Comparing (\ref{eq:coordinate}) and (\ref{eq:coordinate_Q_l}) we obtain a canonical bijection $\tp_l : X_n \to \calO(D_l)$, for each $l=1, \cdots, \ell$, defined by
\begin{equation}\label{eq:bijection component}
([i,m])\tp_l = [m]_{l,i}.
\end{equation}The expanding map $\tp:X_n \to (X_n)^\ell$ is defined by
$$
\tp = \tp_1 \times \cdots \times \tp_\ell.
$$ Note that $\tp$ depends on the choice of diverging points. However components of $\tp$ are all distinct. This follows from that $\calO(D_l)$ does not intersect $\calO(D_{l'})$ if $l\neq l'$.

\begin{Rmk}
A monomorphism $\phi$ of $\calH_n$ determines a group $G_l\leq Sym_n$ which consists of restrictions $g \in\phi(\calH_n)$ on the orbit of $D_l$ under $\phi(\calH_n)$, $1\leq l\leq \ell$.
We remark that $G_l\cong\calH_n$ for all $l$. Observe that each $\calO(D_l)$ coincides with the orbit of $D_l$. This is because all partial translations in $\calP\setminus \calQ_l$ fix $D_l$. Indeed $G_l$ is generated by $n-1$ partial translations of $\calQ_l$. The map $\tp_l$ conjugates $\calH_n$ to $G_l$ in the ambient group $\Sym_n$.
\end{Rmk}

\begin{Prop}\label{prop:prod_transposition}
Suppose $\tau=(P,Q)$ is the transposition exchanging $P$ and $Q$. Then $\E(\phi(\tau)) $ coincides with a product of $\ell$ commuting transpositions $\prod_{j=1}^{\ell} \left((P)\tp_j , (Q)\tp_j\right)$.
\end{Prop}
\begin{proof}
Assume $n\geq 3$.
We consider the following four cases depending on which rays $P$ and $Q$ lie. \\
\textbf{Case I}. Suppose both $P$ and $Q$ lie on $R_1$. Let $P=[2,m]=(1,-m+1)$ and $Q=[2,m']=(1,-m'+1)$ with $m < m'\leq 0$. From the definition of $\tp$ we have
$$
(P)\tp_l =  ([2, m])\tp_l = [m]_{l,2}, \text {  and  }
(Q)\tp_l =  ([2, m'])\tp_l = [m']_{l,2}
$$for each $l$. Since components of $\tp$ are all distinct, $\prod_{l=1}^{\ell} \Bigl((P)\tp_l , (Q)\tp_l\Bigr)=\prod_{l=1}^{\ell} \Bigl([m]_{l,2} , [m']_{l,2}\Bigr)$ is a product of $\ell$ commuting transpositions.

On the other hand, one can check that $\tau$ can be written as $\tau=\alpha ^h$ where
$$
h=g_2 g_3^{m-m'+1} g_2^{m'-1}.
$$ So $\E(\phi(\tau)) = \E(\phi(\alpha^h)) = \E(\phi(\alpha))^{E(\phi(h))}= \E(\phi(\alpha))^{\phi(h)}$. Recall that $\phi(g_2)$ contains partial translations $\pt_{1,2}, \cdots, \pt_{\ell, 2}$ with $D_l = [0]_{l,2}$ for $l=1, \cdots , \ell$. So the involution $\E(\phi(\alpha)) $ is the product of $\ell$ transpositions $\tau_1, \cdots, \tau_\ell$ where $\tau_l$ exchanges $[0]_{l,2}$ and $[-1]_{l,2}$. We examine the effect of $\E(\phi(h))$ on those points. Since $\phi(g_3)$ fixes $[m]_{l,2}$ for all $m\geq 1$, we have, for each $l$,
$$
 [0]_{l,2} \xrightarrow{\phi(g_2)}([0]_{l,2}) \pt_{l,2} = [1]_{l,2}\xrightarrow{\phi(g_3^{m-m'+1})} [1]_{l,2}\xrightarrow{\phi(g_2^{m'-1})}
([1]_{l,2})\pt_{l,2}^{m'-1}=[m']_{l,2},
$$
$$
[-1]_{l,2} \xrightarrow{\phi(g_2)} [0]_{l,2}\xrightarrow{\phi(g_3^{m-m'+1})} ([0]_{l,2})\qt_{l,3}^{m-m'+1}=[m-m'+1]_{l,2}\xrightarrow{\phi(g_2^{m'-1})}
([m-m'+1]_{l,2})\pt_{l,2}^{m'-1}=[m]_{l,2}
$$ where $\qt_{l,3}$ is the partial translation of $\phi(g_3)$ such that $\qt_{l,3}=\pt_{l,2}$ on $\{[m]_{l,2}:m\leq -1\}$. Therefore $\E(\phi(\tau))$ is the product of $\ell$ transpositions exchanging $[m']_{l,2}$ and $[m]_{l,2}$, $l=1, \cdots, \ell$, as expected.

To complete the proof, one can repeat similar calculation for $\E(\phi(h))$ in the following cases:\\
\textbf{Case II}. $P=[2, m]$ and $Q=[i, m']$, with $m\leq 0$, $m'\geq 1$. Then $\tau= \alpha^h$ where
$$
h = \begin{cases}g_2g_3g_2^{m'-1}g_3^{m-1}&i=2\\
                 g_2 g_i^{m'}g_2^{m-1} &3\geq i\leq n.
    \end{cases}
$$
\textbf{Case III}. $P=[i,m]$ and $Q=[i,m']$ with $1\leq m'< m$. Then $\tau= \alpha^h$ where
$$
h = g_2 g_3^{-(m-m'-1)}g_2^{-1}g_i^m.
$$
\textbf{Case IV}. $P=[i,m]$ and $Q=[j,m']$ with $i\neq j$. Then $\tau= \alpha^h$ where
$$
h = g_i g_j g_i^{m-1} g_j^{m'-1}.
$$

If $n=2$ we can apply similar argument with $\alpha=\bigl((1,1),(1,2)\bigr)$. Suppose $P =[2, m]$ and $Q=[2,m']$ with $1\leq m< m'$. The transposition $\tau=(P,Q)$ can abe written as a conjugation $\tau= \alpha^h$ where
$$
h=(g_2 \alpha)^{m'-m-1} g_2 ^{m+1}.
$$
Since $\phi(\alpha)$ exchanges $[-1]_{l,2}$ and $[0]_{l,2}$, but fixes all other points for each $l$ we have
$$
 [0]_{l,2} \xrightarrow{\phi(g_2\alpha)} [1]_{l,2}\xrightarrow{\phi(g_2\alpha)} ([2]_{l,2})\rightarrow \cdots \rightarrow
[m'-m-1]_{l,2}\xrightarrow{\phi(g_2)^{m+1}} [m']_{l,2},
$$
$$
[-1]_{l,2} \xrightarrow{\phi(g_2\alpha)} [-1]_{l,2}\rightarrow \cdots \rightarrow [-1]_{l,2}\xrightarrow{\phi(g_2^{m+1})}
[m]_{l,2}
$$ Therefore $\E(\phi(\tau))= \prod_{l=1}^{\ell} \Bigl([m]_{l,2} , [m']_{l,2}\Bigr)=\prod_{l=1}^{\ell} \Bigl((P)\tp_l ,(Q)\tp_l\Bigr)$. The other cases can be taken care of by similar calculation because a transposition of $\calH_2$ can be expressed as a conjugation of $\alpha$ by $(g_2\alpha)^{m_1}g_2^{m_2}$ for some $m_1, m_2$.
\end{proof}

\begin{Lemma}\label{lm:bound for ell=1}
Suppose that $\phi(g_i) =g_i f_i$ for all $2\leq i \leq n$ where $f_i \in \FSym_n$. If $\tau$ is a transposition, there exists a constant $A_2$, which do not depend on $k$, such that $|\phi^k(\tau)|\leq A_2k$ for all $k\in \bbn$.
\end{Lemma}
\begin{proof}
By the observation (\ref{eq:prod partial translations}), each $\phi(g_i)$ is nothing but a partial translation $\pt_{i}$. So $\pt_{i}$ acts as a translation by $ -1$ on $R_1$ and by $1$ on $R_i$ up to a finite set for $2\leq i\leq n$. We first show that $|\phi^k(\tau)|$ for a transposition $\tau =(P,Q)$. By Proposition \ref{prop:prod_transposition} we see that
$$
\phi^k(\tau)= \bigl((P)\tp^k, (Q)\tp^k\bigr)
$$ where $\tp=\tp_1$ is the expanding map $X_n \to \ES$ defined in (\ref{eq:bijection component}). We claim that both $(P)\tp^k$ and $(Q)\tp^k$ belong to $B_{n,r}$ with $r\leq m+ks$ for all $k\in \bbn$  for some constants $m$ and $s$ which are determined by $\phi$. Let $[p]_{i}$ denote the point $[p]_{\pt_i}$. For each $2\leq i$ one can find the smallest integer $0<k_i $ such that
$$
[p]_{i} = (i, m_{i}+p-k_i)
$$ for all $k_i\leq p$ where $[k_i]_i = (i,m_i)$. We take the largest number $k_1\leq 0$ so that
$$
[p]_{2} = (1, m_{1}+k_1-p)
$$for all $p\leq k_1$ where $[k_1]_2 = (1, m_1)$. Note that $k_1$ is well defined since the actions of $\phi(g_2)^{-1}, \cdots, \phi(g_n)^{-1}$ are identical on $\{[p]_2:p\leq 0\}$. Let $B_{n,m}$ denote the ball of $X_n$ with radius $m=\max\{m_1, \cdots, m_n\}$. Since each $\phi(g_i)$ acts as a translation by $\pm1$ on each ray outside $B_{n,m}$, $P\in B_{n,s}$ implies that $(P)\tp \in B_{n,s+1}$. Observe that if $P=(i,p)\notin B_{n,m}$ then $(P)\tp \in B_{n,p+s}$ where $s=\max\{|m_2-k_2|, \cdots, |m_n-k_n|, |m_1+k_1-1|\}$. One can check that if $k_i \leq p$ then, by the coordinate system (\ref{eq:coordinate}),
$$
([i,p])\tp = [p]_i = (i, m_{i}+p-k_i) = [i, p+(m_i-k_i)]
$$for each $2\leq i$. Similarly if $p\leq k_1$ then
$$
([2,p])\tp = [p]_{2} = (1, m_{1}+k_1-p) = [2, p-(m_1+k_1-1)]
$$ by (\ref{eq:coordinate}). Thus our claim is verified, and so the transposition $\phi^k(\tau)= \bigl((P)\tp^k, (Q)\tp^k\bigr)$ has support in $B_{n,r}$ with $r\leq m+ks$ for all $k\in \bbn$. It is not difficult to show that a transposition $\tau_0$ with $\supp \tau_0\subset  B_{n,r}$ has length $<10r$. Observe that $\tau_0$ can be written as a conjugation $\tau_0=\alpha^h$ where $h\in \calH_n$ can be taken so that $2|h| + |\alpha|< 10r$ as in the proof of Proposition~\ref{prop:prod_transposition}. Therefore we have
$$
|\phi^k(\tau)|= |\bigl((P)\tp^k, (Q)\tp^k\bigr)|< 10(m+ks)<10(m+s)k = A_2k
$$ for all $k$.
\end{proof}

From now on we assume that a monomorphism $\phi$ with $2\leq \ell(\phi)$ satisfies that
$\phi(g_i)$ has a source $R_1$ and a target $R_i$ for all $2\leq i\leq n$. (In view of Proposition \ref{cor:key lemma} this assumption is legitimate) We also assume that $P\in \ES$ unless otherwise stated.\\

\textbf{Rooted trees induced by $\tp$.} The iteration of $\tp$ applied to $P$ determines a \emph{labeled rooted $\ell$-ary tree}, which we will denote by $\calT_P$. The vertex set $V_P$ of $\calT_P$ is equipped with \emph{label, level} and \emph{$\ell$-ary sequences}. The root of $\calT_P$ is the unique vertex at level $0$, labeled by $P$, which corresponds to the empty sequence. 
Inductively the label $L:V_P\to \ES$ and corresponding sequence $W:V_P \to \Omega_{\ell}$ are defined such that a vertex $v\in V_P$ at level $k\in \bbn$ is labeled by
$$
L(v)=  (P)\tp_{l_1} \tp_{l_2} \cdots \tp_{l_k}
$$ and corresponds to a sequence
$$
W(v) = l_1 l_2 \cdots l_k
$$ where $\Omega_{\ell}$ consists of $\ell$-ary sequences on $\{1, \cdots, \ell\}$ and $1\leq l_j\leq \ell$ for all $1\leq j\leq k$. Note that $W(v)$ can be realized as the unique edge path from the root to $v$ with edge labels $\{1, \cdots ,\ell\}$.  Consequently we have a bijective correspondence between $V_P^k$, the set of vertices at level $k$, and $\Omega_{\ell, k}$, the set of all $\ell$-ary sequences of length $k$. 
\begin{Example}\label{ex:example_tree}
Consider a monomorphism $\phi$ of $\calH_3$ defined by $\phi(g_i) = (g_i)^2 \cdot f_i$ for $i=2,3$ where $f_2$ and $f_3$ are cycles given by
$$
f_2 : (1,1) \to (1,2) \to (2, 1) \to (1,1), \quad f_3 : (1,1)\to (1,2) \to (2, 1)\to (3,2)\to (1,1)
$$ From direct computation we see that $\phi(\alpha) = \E (\phi(\alpha))$ is the product of two transpositions $\tau_1$ and $\tau_2$ where
$$
\tau_1 :  \bigl((2,1),(1,4)\bigr), \quad \tau_2 :  \bigl((1,2),(1,3)\bigr).
$$Set $D_1 =(2,1)$ and $ D_2 = (1,2)$. Each of $\phi(g_2)$ and $\phi(g_2)$ has $\ell(\phi) =2$ partial translations which are
$$
\pt_{1,2} : \cdots \to (1, 2m+4) \to (1, 2m+2)\to \cdots \to(1,4) \to (2,1) \to\cdots \to (2, 2m-1)\to (2, 2m+1) \to \cdots
$$
$$\pt_{2,2} : \cdots \to (1, 2m+3) \to (1, 2m+1)\to \cdots \to(1,3) \to (1,2)\to(1,1) \to\cdots \to (2, 2m)\to (2, 2m+2) \to \cdots
$$
$$
\pt_{1,3} : \cdots \to (1, 2m+4) \to (1, 2m+2)\to \cdots \to(1,4) \to (2,1) \to\cdots \to (3, 2m)\to (3, 2m+2) \to \cdots \quad \quad
$$
$$\pt_{2,3} : \cdots \to (1, 2m+3) \to (1, 2m+1)\to \cdots \to(1,3) \to (1,2) \to\cdots \to (3, 2m-1)\to (3, 2m+1) \to \cdots
$$ where $m \in \bbn$. With $D_l=[0]_{l,2}= [0]_{l,3}$ for $l=1,2$, let us apply $\tp$ repeatedly to describe rooted trees of points in $\{(1,1), (1,2), (2,1)\}=\supp f_2$.

\noindent\begin{tikzpicture}[every tree node/.style,
   level distance=1.8cm,sibling distance=0cm,
   edge from parent path={(\tikzparentnode) -- (\tikzchildnode)}]
\Tree [.\node {$(1,1)=[2,0]$ };
        \edge node[auto=right] {$\tp_1$};
        [.{$[0]_{1,2} =D_1 = (2,1) = [2,1]$}
         \edge node[left] {$\tp_1\;\:$};
          [.{$[1]_{1,2}\!=(2,3)=[2,3]$}
          [.{$[3]_{1,2}\! = [2,7]$} ]
          [.{$[3]_{2,2}\! = [2,4]$} ] ]
          \edge node[right] {$\;\:\tp_2$};
          [.{$[1]_{2,2}\!=(1,1)=[2,0]$}
            [.{$[2,1]$} ] [.{$[2,-1]$} ]
          ]
        ]
       \edge node[auto=left] {$\tp_2$};
      [.{$[0]_{2,2}=D_2 = (1,2) = [2,-1]$}
         \edge node[left] {$\tp_1\;\:$};
         [.{$[-1]_{1,2}\!=(1,4)=[2,-3]$}
         [.{$[-3]_{1,2} \!= [2,-7]$} ]
         [.{$[-3]_{2,2}\! = [2,-6]$} ]]
             \edge node[right] {$\;\:\tp_2$};
         [.{$[-1]_{2,2}\!=(1,3)=[2,-2]$}
          [.{$[2,-5]$} ] [.{$[2,-4]$} ]  ]
      ]
    ]
\end{tikzpicture}
The above illustrates the tree $\calT_{(1,1)}$ with labels up to level $3$. Each vertex $v$ has $\ell=2$ children; the left child is $(v)\tp_1$ and the right child is $(v)\tp_2$. Accordingly all left edges are labelled by $1$ ($\tp_1$) and all right edges are labeled by 2 ($\tp_2$). Each $v\in V_{(1,1)}^k$ corresponds to a unique sequence $W(v)\in \Omega_{2,k}$. For example the vertex $[2,0]$ at level $2$ corresponds to the sequence $12 \in \Omega_{2,2}$. One important observation is that $\calT_{(1,1)}$ contains infinitely many copies of itself. The point $(1,1)$ appears as a label for all vertex $v$ with $W(v)=12 12 \cdots 12$ (as emphasized in blue color!) because the subtree spanned by $v$ and its descendants is identical to $\calT_{(1,1)}$.

The following figure illustrates the tree $\calT_{(1,2)}$ up to level $3$, which happens to be a subtree of $\calT_{(1,1)}$. Observe that $\calT_{(1,2)}$ never contains a copy of itself. This is because the vertices of the tree are all distinct. Yet another crucial fact is that each pair of children has labels which are \emph{translations} of the labels $(1,4)$ and $(1,3)$ at level $1$. For example, $\calT_{(1,2)}$ has four pairs of labels at level $3$, each of which is a translation of the pair $(1,4)$ and $(1,3)$ by some power of $g_2$. Intuitively this is because $\pt_{1,2}$ and $\pt_{2,2}$ act as a translation by $\pm2$ on all points in their supports but finitely many. More precisely
$$
([2,m])\tp_1=[m]_{1,2} = \begin{cases} [2,2m-1] = (1,-2m+2) & m\leq -1\\
                        [2,2m+1] = (2,2m+1) & 1\leq m\\
            \end{cases},
$$
$$
([2,m])\tp_2=[m]_{2,2} = \begin{cases} [2,2m] = (1,-2m+1) & m\leq -1\\
                        [2,2m-2] = (2,2m-2) & 2\leq m\\
            \end{cases}
$$where the last equalities in each case follow from
$$
[2,m]=\begin{cases} (1,-m+1) & m\leq 0\\
                    (2,m) & 1\leq m
                    \end{cases}
$$by (\ref{eq:coordinate}). See Proposition~\ref{prop:stable points}.

\noindent\begin{tikzpicture}[every tree node/.style,
   level distance=1.8cm,sibling distance=.2cm,
   edge from parent path={(\tikzparentnode) -- (\tikzchildnode)}]
\Tree [.\node  {$D_2 = (1,2) = [2,-1]$};
         \edge node[auto=right] {$\tp_1$};
        [.{$[-1]_{1,2}\!=(1,4)=[2,-3]$}
            [.{$[-3]_{1,2} \!= [2,-7]$}
                [.{$[-7]_{1,2} \!= (1,16)$} ]
                [.{$[-7]_{2,2}\! = (1,15)$} ]
            ]
            [.{$[-3]_{2,2}\! = [2,-6]$}
                [.{$(1,14)$} ]
                [.{$(1,13)$} ]
            ]
        ]
      \edge node[auto=left] {$\tp_2$};
         [.{$[-1]_{2,2}\!=(1,3)=[2,-2]$}
            [.{$[-2]_{1,2} \!= [2,-5]$}
                [.{$(1,12)$} ]
                [.{$(1,11)$} ]
            ]
            [.{$[-2]_{2,2}\! = [2,-4]$}
                [.{$(1,10)$} ]
                [.{$(1,9)$} ]
            ]
        ]
      ]
\end{tikzpicture}

The tree $\calT_{(2,1)}$ below also contains infinitely many copies of itself. The subtrees determined by vertices with sequences $21 21\cdots 21$ are identical to $\calT_{(2,1)}$. On the other hand, the subtree $\calT_{(2,3)}$ illustrates the opposite behavior; it does not contain a copy of itself. Moreover any children at level $k\geq 2$ in $\calT_{(2,3)}$ has labels which are translations of the labels $(2,7)$ and $(2,4)$.

\noindent\begin{tikzpicture}[every tree node/.style,
   level distance=1.8cm,sibling distance=.1cm,
   edge from parent path={(\tikzparentnode) -- (\tikzchildnode)}]
\Tree [.\node {$(2,1) = [2,1]$ };
   \edge node[auto=right] {$\tp_1$};
          [.{$[1]_{1,2}\!=(2,3)=[2,3]$}
             [.{$[3]_{1,2}\! = (2,7)=[2,7]$}
                [.{$[2,15]\!=\!(2,15)$} ] [.{$[2,12]\!=\!(2,12)$} ]]
             [.{$[3]_{2,2}\! =(2,4)= [2,4]$}
                [.{$[2,9]\!=\!(2,9)$} ] [.{$[2,6]\!=\!(2,6)$} ]
             ]
          ]
             \edge node[auto=left] {$\tp_2$};
         [.{$[1]_{2,2}\!=(1,1)=[2,0]$}
              [.{$(2,1) = [2,1]$}
                [.{$(2,3)$} ]
                [.{$(1,1)$} ]
              ]
      [.{$D_2 = (1,2) = [2,-1]$}
         [.{$(1,4)$} ]
         [.{$(1,3)$}  ]
       ]
        ]
    ]
\end{tikzpicture}
\end{Example}

For a vertex $v\in V_P$ with $\omega=W(v)$, let $P_\omega$ denote the label $L(v)$. From Proposition \ref{prop:prod_transposition} we have the following.

\begin{Cor}\label{coro:prod_transposition}
Suppose $\tau=(P,Q)$ is the transposition on two points $P, Q \in \ES$. Then $\E(\phi^k(\tau))$ coincides with a product of $\ell^k$ commuting transpositions $\prod_{\omega\in \Omega_{\ell,k}} \left(P_\omega, Q_\omega\right)$ for each $k\in \bbn$.
\end{Cor}
\begin{proof}
The base case follows directly from Proposition \ref{prop:prod_transposition} since $\Omega_{\ell,1}=\{1, \cdots, \ell\}$. Suppose that $\E(\phi^k(\tau)) = \prod_{\omega\in \Omega_{\ell,k}} \tau_\omega$ is a product of $\ell^k$ commuting transpositions for $k\geq 1$ where $\tau_\omega=(P_\omega, Q_\omega)$. By Proposition \ref{prop:prod_transposition} again, we have
$$
\E\bigr(\phi(\tau_\omega)\bigr) = \prod_{j=1}^{\ell} \left((P_\omega)\tp_j , (Q)\omega)\tp_j\right)
$$ for all $\omega \in \Omega_{\ell, k}$. So $\E\bigr(\phi(\tau_\omega)\bigr)=\prod_{\omega'} \left(P_{\omega'}, Q_{\omega'}\right)$ where $\omega'=\omega l\in \Omega_{\ell, k+1}$ for $l=1, \cdots, \ell$. Therefore
\begin{align*}
\E\bigr(\phi^{k+1}(\tau)\bigr)
&=\E\bigr(\phi \phi^{k}(\tau)\bigr)
=\E\Bigl(\phi\bigl(\prod_{\omega\in \Omega_{\ell,k}} \tau_\omega\bigr)\Bigr)
=\prod_{\omega\in \Omega_{\ell,k}} \Bigl(\E\bigl(\phi(\tau_\omega)\bigr)\Bigr)\\
&= \prod_{\omega\in \Omega_{\ell,k}}\prod_{j=1}^{\ell}\Bigl((P_\omega)\tp_j , (Q_\omega)\tp_j\Bigr)
=\prod_{\omega\in \Omega_{\ell,k}} \prod_{\substack{\omega'=\omega l\\1\leq l\leq \ell}} \left(P_{\omega'}, Q_{\omega'}\right)=\prod_{\omega\in \Omega_{\ell,k+1}} \left(P_\omega, Q_\omega\right).
\end{align*}
The commutativity of transpositions above follows from Proposition~\ref{prop:distinct vertices level k}.
\end{proof}

Recall the notation $V_P^k$ which denotes the set of all vertices of $\calT_P$ at level $k$.
\begin{Prop}\label{prop:distinct vertices level k}
 The label map $L:V_P^k \to \ES$ is injective for each $k\in \bbn$.
\end{Prop}
\begin{proof}Induction on $k$. The decomposition $\ES = \sqcup_{l=1}^\ell \calO(D_l)$ shows that if $l\neq l'$, $(X_n)\tp_l= \calO(D_l)$ does not intersect $(X_n)\tp_{l'}=\calO(D_{l'})$. This establishes the base case. Suppose $v\in V_P^{k+1}$. Observe that the last letter of $W(v)$ determines which $\calO(D_l)$ the label $L(v)$ belongs to. More precisely, $L(v)\in \calO(D_l)$ if and only if the last letter of $W(v)$ is $l$. If $L(v)=L(v')=Q \in \calO(D_l)$ for vertices $v$ and $v'$ at level $k+1$, then $W(v)$ and $W(v')$ share the same last letter $l$, that is, $W(v)= \eta l$ and $W(v')=\eta'l$ for some $\eta , \eta'\in \Omega_{\ell, k}$. Since $\tp_l:X_n \to \calO(D_l)$ is a bijection, the two vertices which correspond to $\eta$ and $\eta'$ respectively share the same label $(Q)\tp_l^{-1}$. However this contradicts induction assumption.
\end{proof}

\textbf{Stable points.} In Example~\ref{ex:example_tree} we saw that the trees $\calT_{(2,1)}$ and $\calT_{(1,1)}$ contain copies of themselves. However the tree $\calT_{(1,2)}$ illustrates the opposite behavior. Intuitively this is because each $\tp_l$ applied to $(i,k)$ `doubles' the second coordinate for all points of its support but finitely many. We need to formulate this rigorously.

Let $\ell \geq 2$. We want to define \emph{stable points} for $\phi(g_i)$ with the assumption that $\phi(g_i)$ has a unique target $R_i$ for $i=2, \cdots, n$. Fix $i$. Recall that $\phi(g_i)$ translates points of $R_i$ by $+\ell$ up to a finite set.
In other words, the action of $\phi(g_i)$ on $R_i$ is \emph{eventually stabilized} as a translation by $+\ell$. Being a restriction of $\phi(g_i)$, $\pt_{l,i}$ is also eventually stabilized for all $1\leq l\leq \ell$. From property (\ref{eq:source_to_target}) one can take an integer $0\leq k_{l,i}$ for each $l$ such that if $[k]_{l,i}=(i,m)$ then $[k+1]_{l,i}=(i, m+\ell)$ for all $k_{l,i}\leq k$. More precisely, for each $l$, there exists a smallest integer $k_{l,i} $ such that
\begin{equation}\label{eq:stable point}
[k]_{l, i} = (i, m_{l,i}+\ell(k- k_{l,i}))
\end{equation} for all $k_{l,i}\leq k$ where $[k_{l,i}]_{l,i} = (i,m_{l,i})$. The $i^{th}$ \emph{threshold} is the positive integer $s_{i}$ defined by
\begin{equation}\label{eq:threshold}
s_{i} =\displaystyle{\max_{1\leq l \leq \ell}}\Bigl\{\left\lceil\dfrac{\ell k_{l,i} -m_{l,i}}{\ell-1}\right\rceil, k_{l,i}\Bigr\}
\end{equation} where $\lceil \ast \rceil$ stands for the smallest integer function.
A point $(i,k)$ is called a \emph{stable point} if $s_i< k$ where $2\leq i\leq n$. Recall the coordinate system (\ref{eq:coordinate}); $[i,k] = (i,k)$ for $2\leq i\leq n$ and $k\in \bbn$. One crucial observation is that if $[i,k]$ is a stable point, $([i,k])\tp_l$ is again a stable point for all $l$. A stable point $[i,k]$ satisfies $k_{l,i}  < k$ for all $l$. By (\ref{eq:stable point}) we have
\begin{equation}\label{eq:stable point 2}
([i,k])\tp_l = [k]_{l,i} =  (i, m_{l,i}+\ell(k- k_{l,i})).
\end{equation} for all $k >s_i$. Since
$$
k <m_{l,i}+\ell(k- k_{l,i}) \Leftrightarrow  \dfrac{\ell k_{l,i} -m_{l,i}}{\ell-1} < k,
$$ $s_i <k< m_{l,i}+\ell(k- k_{l,i})$ provided $s_{i} < k$. So $([i,k])\tp_l$ is a stable point if $[i,k]$ is a stable point for $2\leq i\leq n$.

We can define stables points on $R_1$ in a similar manner because all the actions of $\phi(g_i)$'s on $R_1$ are also eventually stabilized as a translation by $-\ell$. However we need to change signs and take reverse inequalities accordingly. It suffices to consider one partial translation, say $\pt_{l,2}$, of $\phi(g_2)$ since actions of $\pt_{l,2}^{-1}$ and $\pt_{l,i}^{-1}$ are identical on $\{[k]_{l,2} : k\leq 0\}$ for $3\leq i\leq n$. There exists a largest integer $k_l \leq 0$ such that
\begin{equation}\label{eq:stable point_2}
[k]_{l, 2} = (1, m_l+\ell(k_l-k))
\end{equation} for all $k\leq k_l$ where $[k_{1,2}]_l = (1,m_l)$. Now points $(1,k)\in R_1 $ are called \emph{stable points} for all $ s_1< k$ where $s_1$ is the positive integer defined by
\begin{equation}\label{eq:threshold 1}
s_1 =\max \Bigl\{\left\lceil\dfrac{\ell- \ell k_{l} -m_{l}}{\ell-1}\right\rceil, -k_{l}+1\Bigr\}
\end{equation}
The above definition implies that the image of $(1,k)$ under $\tp_l$ is a stable point if $(i,k)= [2,-k+1]$ (by the coordinate system(\ref{eq:coordinate})) is a stable point.
Since $s_1 < k \Leftrightarrow -k+1< -s_1+1 < k_l$ we have, by (\ref{eq:stable point_2}),
$$
([2,-k+1])\tp_l = [-k+1]_{l,2} =\bigl(1, m_l+\ell(k_l-(-k+1))\bigr)= \left(1, m_l+\ell(k_l+k-1)\right)
$$ for all stable points on $R_1$. So we have $s_1< k<m_l+\ell(k_l+k-1)$ from
$$
k < m_l+\ell(k_l+k-1) \Leftrightarrow  \dfrac{\ell- \ell k_{l} -m_{l}}{\ell-1} < k.
$$
Let $\calS$ denote the set of all stable points. We summarize the discussion above as
\begin{Rmk}\label{rmk:stable points}
If $P=(i,p)$ is a stable point then $(P)\tp_l = (i,q)$ with $p<q$ for each $1\leq l\leq \ell$, $1\leq i\leq n$. In particular, $\calS$ is invariant under $\tp_l$ for each $l$.
\end{Rmk}

\textbf{Translations on a ray.} Let $v\in V_P$ with $\omega = W(v)$. The \emph{descendants} of $v$ at a depth $m\in \bbn$ consists of vertices $u$ such that $W(u) = \omega \omega_1$ for some $\omega_1 \in \Omega_{\ell, m}$. Let $D_{v, m}$ denote the set of descendants of $v$ with depth $m$, and let $D_v= \cup_m  D_{v,m}$. The \emph{children} of $v$ is the descendants of $v$ at depth $1$. Note that $V_P^{k} = \sqcup_{v\in V_P^{k-1}} D_{v,1}$ for all $k$. In case
an ordered $\ell$-tuple $L(D_{v,1})=\bigl( (i,p_1), \cdots, (i,p_\ell)\bigr)$  we suppress the first coordinate and take a vector $\nu = [p_1, \cdots, p_\ell] \in \bbn^\ell$ to express $L(D_{v,1})$. The \emph{translation} of an $\ell$-tuple $\nu$ by $t\in \bbz$ is the $\ell$-tuple
$$
\nu+t=[p_1+t, \cdots, p_\ell+t].
$$ We require that all points of $\nu$ and its translation $\nu +t$ stay in the same ray $R_i$. Let $B_{n,r}$ denote the ball of $X_n$ with radius $r$,
$$
B_{n,r} = \{(i,p)| p\leq r\}.
$$ Recall that $\phi(g_i)$ acts as a translation on $\{(i,p)| p>s_i\}$ and $\{(1,q)| q>s_1\}$ by $\ell$ and $-\ell$ respectively for all $2\leq i\leq n$ where $s_i$ is the $i^{th}$ threshold defined in (\ref{eq:threshold}) and (\ref{eq:threshold 1}). Consequently $\phi(g_i)$ acts by the same manner on $R_i \setminus B_{n,s}$ and $R_1 \setminus B_{n,s}$ for
$$
s=\max\{s_1, \cdots, s_n\}.
$$

\begin{Prop}\label{prop:stable points}
Suppose $P\in \calS\cap R_i$, $1\leq i\le n$. The label map $L:V_P\to \ES$ is injective and $L(V_P) \subset \calS \cap R_i$. For each $v\in V_P$, $L(D_{v,1})$ is a translation of $L(V_P^1)$. Moreover there exists a constant $A_0 = A_0(P)$ such that $L(V_P^k) \subset B_{n, r}$ with $r=A_0 \ell^k +s$
\end{Prop}

\begin{proof}
First we consider the case when $P=(i,p)$ is a stable point for $2\leq i$.
By Remark~\ref{rmk:stable points}, we have $L(v)=(i,q)$ with $p<q$ for all $v\in V_P$, and so $L(V_P) \subset \calS \cap R_i$. Suppressing $i$ in the first coordinate we can write $L(V_P^1) =[p_1, \cdots, p_\ell]$ where
\begin{equation*}\label{eq:stable_point_q}
(i,p_l)=([i,p])\tp_l= [p]_{l,i}= \bigl(i,m_{l,i} + \ell(p-k_{l,i})\bigr)
\end{equation*} for $l=1, \cdots, \ell$ by (\ref{eq:stable point 2}). The label map $L$ restricted on the root vertex and $V_P^1$ is injective since $p<p_l$ for all $l$ and each $p_l$ belongs to $\calO(D_l)$ which are all disjoint for $l=1, \cdots, \ell$. With this base case assume that the map $L$ is injective on the set of vertices up to level $k$. Since no vertex at level $k+1$ attains the label $(i, p)=P$ it suffices to check whether two vertices $v$ and $v'$ (other than the root vertex) at levels $\leq k+1$ share the same label $L(v)= L(v')$. We can apply similar argument as in Proposition~\ref{prop:distinct vertices level k} to draw a contradiction; $L(v)= L(v')$ implies that the ascendants of $v$ and $v'$ share the same label at levels $\leq k$.

For the second assertion suppose $v\in V_P$ with $L(v)=(i,q)$. Since $(1,q)$ is a stable point the identity (\ref{eq:stable point}) implies that $L(D_{v,1})=[q_1, \cdots, q_\ell]$ where
\begin{equation*}\label{eq:stable_point_q}
(i,q_l)= \Bigl(i,m_{l,i} + \ell(q-k_{l,i})\Bigr)= \Bigl(i,m_{l,i} + \ell(p-k_{l,i})+\ell(q-p)\Bigr)
\end{equation*} for $l=1, \cdots, \ell$. Since $p<q$
\begin{equation}\label{eq:translation by q-p}
L(D_{v,1}) =L(V_P^1)+\ell(q-p)
\end{equation} is the translation of $L(V_P^1)$ by $\ell(q-p)>0$.

Since $L(v)=(i,q)$ for all $v\in V_P^k$, we can consider a sequence of natural numbers $\{a_k\} $ so that $a_k$ denotes the maximum of such $q$'s. We want to show $a_k \leq s \ell^k +s$ for all $k$. From the rewriting
\begin{equation}\label{eq:rewriting1}
m_{l,i}+\ell(p-k_{l,i}) =\ell\Bigl(p- \dfrac{\ell k_{l,i}-m_{l,i}}{\ell-1}\Bigr) +\dfrac{\ell k_{l,i}-m_{l,i}}{\ell-1}
\end{equation} we have $a_1\leq A_0\ell +s$ where
$$
A_0 = \max_{1\leq l \leq \ell}\Bigl\{p- \dfrac{\ell k_{l,i}-m_{l,i}}{\ell-1}\Bigr\}.
$$
(since $P=(i,p)$ is a stable point the above maximum is taken over positive numbers, and so $A_0>0$.)
One can rewrite $a_{k+1}= m_{l,i}+\ell(a_k-k_{l,i})$, which follows from the second step above,  as in (\ref{eq:rewriting1}) with $p$ replaced by $a_k$ to check that
\begin{equation}\label{eq:recursive}
a_{k+1}- s= \ell (a_k -s)
\end{equation} for all $k\in \bbn$. Therefore $L(V_P^k) \subset B_{n,a_k}$ and $a_k \leq A_0\ell^k +s$ for all $k\in \bbn$.

In case $P=(1,p)$ one can apply analogous arguments. Remark~\ref{rmk:stable points} together with `cancelling argument' shows that $L(v)= (1,q)$ is a stable point for all $v\in V_P$. By $(1,q) = [2, -q+1]$ and (\ref{eq:stable point_2}), we have that $((1,q))\tp_l$ becomes
$$
([2,-q+1])\tp_l = [-q+1]_{l,2} = \Bigl(1, m_l+\ell(k_l+q-1)\Bigr) = \Bigl(1, m_l+\ell(k_l+p-1)+\ell(q-p)\Bigr)
$$ for all $l$. Therefore $L(D_{v,1})$ is a translation of $L(V_P^1)$ by $\ell(q-p)>0$. From the rewriting for the root $P=(1,p)=[2,-p+1]$,
\begin{equation*}\label{eq:rewriting}
([2,-p+1])\tp_l =m_{l}+\ell(k_{l}+p-1) =\ell\Bigl(p- \dfrac{\ell -m_{l}-\ell k_{l}}{\ell-1}\Bigr) +\dfrac{\ell -m_{l}-\ell k_{l}}{\ell-1}
\end{equation*} we find $a_1 = A_0\ell+s$ such that $L(V_P^1)\subset B_{n,a_1}$ where
$$
A_0 = \max_{1\leq l \leq \ell}\Bigl\{p- \dfrac{\ell -m_{l}-\ell k_{l}}{\ell-1}\Bigr\}.
$$Finally the identity $a_{k+1} =m_{l}+\ell(k_{l}+a_k-1)$ implies that $L(V_P^k)\subset B_{n,a_k}$ where the sequence $\{a_k\}$ satisfies (\ref{eq:recursive}) for all $k\in \bbn$.
\end{proof}

\begin{Rmk}\label{bound on the ball}
We remark that the radii of balls that contain $L(V_P^k)$ can be bound by a linear term of $\ell^k$ even when $P$ is not a stable point. Intuitively this happens because the images of $P$ under expanding map $\tp^{\,k}$ travel inside the ball $B_{n,s}$ for first finite steps. Only after does $L(V_P^k)$ contain stable points the radii of the balls contain $V_P^k$ follow the growth as described in Proposition~\ref{prop:stable points}. More precisely, a ball $B_{n,a_k}$ that contains $L(V_P^k)$ has radius
$$
a_k\leq \begin{cases} s &k\leq k_0\\
                    A_0 \ell^{k-k_0}+s & k_0+1\leq k\\
        \end{cases}
$$for some constant $A_0$.
In all, every $P\in \ES$ determines a constant $A_0$ such that
 $L(V_P^k)\subset B_{n,r}$ with $r\leq A_0\ell^k +s$.
\end{Rmk}

\textbf{Intervals of $V_P^k$.}
We want to decompose $V_P^k$ into \emph{intervals} which provide a coarser decomposition than $V_P^k=\sqcup_{v\in V_P^{k-1}} D_{v,1}$. Motivating examples come from stable points. Proposition \ref{prop:stable points} states that $L(D_{v,1})$ is a translation of $L(V_P^1)$ for all $v\in V_P^{k-1}$ if $P$ is a stable point. The set $\{L(D_{v,1}): v\in V_P^{k-1}\}$ can be viewed as a single orbit of $V_P^1$ under translation. In this case, we want to take $V_P^k$ as a single interval.

Let $v\in V_P^{k_1}$ with $k_1<k$. We say that $D_{v, k-k_1}\subset V_P^k$ is an \emph{interval} of $V_P^k$ if either $v\in V_P^{k-1}$ or $L(v)$ is a stable point. (The naming `interval' comes from that if $D_{v, k-k_1}$ is an interval, $\Lambda(D_{v, k-k_1})$ forms an interval of $\bbn$ with size $\ell^{k-k_1}$ where $\Lambda: V_P^k \to \Omega_{\ell, k}\to\{1, \cdots, \ell^k\}$ is the unique map such that $\Lambda(v)=j$ if $W(v)$ is the $j^{th}$ sequence in the lexicographic order on $\Omega_{\ell,k}$.) Note that an interval may contain smaller intervals. For example, a descendant $u$ of $v$ determines an interval $D_u \cap V_P^{k}$ if $u\in V_P^{k_2}$ with $k_2 <k$. However we can take maximal intervals so that $V_P^k$ consists of minimum number of intervals. For each $k\in \bbn$, let $N_P(k)$ denote the minimum number of intervals which cover $V_P^k$. We remark that $V_P^k$ can be expressed as a disjoint union of $N_P(k)$ intervals.

In view of Proposition \ref{prop:stable points}, we say a tree $\calT_P$ is \emph{stabilized} if $L:V_P\to \ES$ is injective. Let $\calU \subset \ES$ denote the set of all points $P $ such that $L:V_P\to \ES$ is not injective. Proposition \ref{prop:stable points} implies that the cardinality of $\calU$ is finite since $\calU \subset \ES \setminus \calS$. (example)

\begin{Lemma}\label{lm:bound on interval}
There exists a constant $A_1=A_1(\phi)$ such that
$N_P(k)\leq (k-1)(\ell-1)A_1+1$ for all $P\in \ES$ and $k\in \bbn$.
\end{Lemma}
\begin{proof}
Proposition~\ref{prop:stable points} implies that $N_P(k)=1$ for all $P\in \calS$. Indeed
$N_P(k)$ is bounded by a constant for all $P\notin \calU$. If $L:V_P\to \ES$ is injective, one can find $k_0 \in \bbn$ such that $L(V_P^k)$ consists of stable points for all $k\geq k_0$ since $|\ES \setminus \calS|$ is finite. From the decomposition
$$
V_P^k = \bigsqcup_{u\in V_P^{k_0}} D_{u, k-k_0},
$$ where each component is an interval of $V_P^k$, we have $N_P(k) \leq |V_P^{k_0}|=\ell^{k_0}$ for all $k\geq k_0$.
So $N_P(k) \leq \ell^{k_0}$ for all $k\in \bbn$ since the function $N_P(k)$ is monotone increasing on $k$. Taking maximum of upper bounds $\ell^{k_0}$ over all $P\in \ES\setminus \calS \setminus \calU$, which is a finite set, one obtains a constant $A_1\geq 1$ such that $N_P(k)\leq A_1$ for all $P\in \ES\setminus \calU$. To establish desired bound for points of $\calU$ we need following steps together with induction on $k$.

\textbf{Step 1}. A vertex $v$ of $\calT_P$ and its descendants determine a unique subtree $\calT_v$. Note that a pair of subtrees $\calT_v$ and $\calT_{v'}$ satisfies the following dichotomy: they do not intersect or one contains the other. In this step we show that if $L(v) = L(v')$ then either $\calT_v \subset \calT_{v'}$ or  $\calT_{v'} \subset \calT_{v}$.

Suppose $L(v) = L(v')$ for $v\neq v'$. Then $W(v)=\omega l$ and $W(v')=\omega'l$ for some $l$. As in proof of Proposition \ref{prop:distinct vertices level k}, we can cancel the last same letters of two sequences. We have either $\omega l$ is a subword of $\omega' l$ or vice versa since $v\neq v'$. This precisely means that one subtree contains the other.

\textbf{Step 2}. We claim that $P\in \calU$ if and only if $\calT_P$ contains a vertex $v$ other than the root with $L(v)=P$. We need to show that if the label map $L:\calT_P\to \ES$ is not injective then $V_P$ contains $v$ at level $k\geq 1$ with $L(v)=P$. Suppose $V_P$ contains $v_1$ and $v_2$ with $L(v_1) = L(v_2) =Q$. From Step 1 we know that one of subtrees contains the other. We may further assume that one of two trees, say $\calT_{v_1}$, is maximal, which means that $\calT_{v_1}$ contains all subtrees determined by $v$ with $L(v) =Q$. Consequently $\omega_1 = W(v_1)$ is a subsequence of $\omega_2 = W(v_2)$, i.e., $\omega_2 =\omega_1 \eta$ for some sequence $\eta$. The first case to consider is when $|\omega_1|\leq |\eta|$. Using the same canceling argument as in Step 1, one can show that $\eta=\eta_0 \omega_1$ for some sequence $\eta_0$. The sequence $\omega_1= l_1 \cdots l_k$ determines a composition of injective maps such that $(P)\tp_{l_1} \cdots \tp_{l_k}=Q$. Observe that $\omega_1$ also determines a unique path from the vertex $u$, which corresponds to $\omega_1 \eta_0$, to $v_2$. This path transforms into the same composition of maps such that $(L(u))\tp_{l_1} \cdots \tp_{l_k}=L(v_2)=Q$. So $u$ is a vertex with $L(u)=P$. The next case to consider is when $|\omega_1|>|\eta|$. The cancelling argument implies that $\omega_1$ is the concatenation $\omega_1 = \omega_0 \eta$ for some sequence $\omega_0$. Since the sequence $\eta= l'_1 \cdots l'_{m}$ determines identity map $(Q)\tp_{l'_1} \cdots \tp_{l'_m}=(L(v_1))\tp_{l'_1} \cdots \tp_{l'_m}=L(v_2)=Q$ we have $(L(u))\tp_{l'_1} \cdots \tp_{l'_m}=L(v_1)=Q$ for the vertex $u$ corresponding to $\omega_0$. So $L(u)=Q$. However this contradicts the maximality of the tree $\calT_{v_1}$.

\textbf{Step 3}. Next we show that $V_P^1$ contains only one vertex $v$ with $L(v)\in \calU$. Suppose there exist two vertices $v_1$ and $v_2$ at level $1$ such that $L(v_i) \in \calU$, $i=1,2$. By Step $2$, $\calT_{v_i}$ contains a vertex $u_i$ with $L(u_i)=L(v_i)$, $i=1,2$.
Suppose $L(v_1)= (P)\tp_l$ for some $l$. Since $L(u_1)=L(v_1)$ the parent of $u_1$ has label $P$. So $\calT_{v_1}$ contains a vertex $w_1$ such that $L(w_1)=P$. By the same reason $\calT_{v_2}$ contains a vertex $w_2$ with $L(w_2)=P$. Since two subtrees $\calT_{v_1}$ and $\calT_{v_2}$ do not intersect, $\calT_P$ contains two disjoint copies of itself which are determined by $w_1$ and $w_2$. Obviously this can not occur by Step $1$.

\textbf{Step 4}. We complete the proof by induction on $k$. The base case is obvious since $N_P^1=1$ for any $P\in \ES$. From Step $3$, we know that $V_P^1=\{v_1, \cdots, v_\ell\}$ contains one vertex $v_j$ with $L(v_j)=Q_j \in \calU$. Since all other subtrees $\calT_{v_l}$ are stabilized there exists a constant $A_1$ such that $N_{Q_l}(k-1) \leq A_1$ for all $Q_l= L(v_l)$ and $k\in \bbn$, $i\neq j$. We complete the proof by
$$
N_P(k) \leq \sum_{1\leq l\leq \ell}A_1 {Q_l}(k-1)\leq A_1 {Q_j}(k-1) +(\ell-1) A_1\leq(k-1)(\ell-1)A_1+1.
$$
\end{proof}

\begin{Lemma}\label{lm:bound on transpositions}
Suppose $\tau=(P, Q)$ is the transposition on two points $P, Q\in \ES$.
There exists a polynomial $\rho$ on $k \in \bbn$ whose degree does not depend on $k$ such that $|\phi^k(\tau)|\leq \rho(k)\ell^k$ for all $k\in \bbn$.
\end{Lemma}
\begin{proof}
First we show the existence of such a polynomial when $P$ and $Q$ are stable points, and then we extend the discussion to the general case.
Suppose $P=(i,p)$ is a stable point. Observe that explicit expressions for $L(v)$ can be obtained for $v\in V_P^{k}$ whenever $W(v)$ is given. We want to find such a expressions using $L(V_P^1)$.
For each $1\leq l\leq \ell$, let $V_P^1=\{v_1, \cdots, v_\ell\}$ with $L(v_l)= (i,p_l) = (P)\tp_l$, and let $d_l = p_l-p$. We claim that if $W(v)= l_1 l_2 \cdots l_{k+1} \in \Omega_{\ell, k+1}$ then $L(v) = (i,p_v)$ where
\begin{equation}\label{eq:p_v}
 p_v= d_{l_1} \ell^{k} + d_{l_2} \ell^{k-1} +\cdots + d_{l_{k}} \ell + p_{l_{k+1}}
\end{equation}
for each $k\in \bbn$. The base case for $k=1$ follows immediately from (\ref{eq:translation by q-p}). Assume $u\in V_P^{k+2}$ with $W(u)= l_1 l_2 \cdots l_{k+2}$. The unique parent vertex $v$ of $u$ with $W(v) =l_1 l_2 \cdots l_{k+1}$ has label $L(v) = (i,p_v)$ which satisfies (\ref{eq:p_v}). By the identity (\ref{eq:translation by q-p}), $L(u) = (i,p_u)$ satisfies
$$
p_u = \ell(p_v-p)+p_{l_{k+2}} = \ell\bigl(d_{l_1} \ell^{k} +\cdots + d_{l_{k}} \ell + p_{l_{k+1}}-p\bigr) +p_{l_{k+2}}= d_{l_1} \ell^{k+1} +\cdots + d_{l_{k+1}} \ell + p_{l_{k+2}}
$$since $u$ is the $(l_{k_2})^{th}$ child of $v$.

Next we want to describe the relationship between $L(D_{v_l, k})$ and $L(V_P^k)$ for $1\leq l\leq \ell$ and $k\in \bbn$. A vertex $v_l\in V_P^1$ determines a set of vertices $D_{v_l, k} \subset V_P^{k+1}$ with $|D_{v_l, k}|=\ell^k =|V_P^k|$. Observe that each vertex $u\in D_{v_l, k}$ corresponds to $W(u)= l \omega$ for some $\omega \in \Omega_{\ell, k}$. The concatenation $W(v)\mapsto l W(v)$ induces a bijection between $V_P^{k}$ and $D_{v_l, k}$; $V_P^{k}\ni v \leftrightarrow u \in D_{v_l, k}$ if $W(u)=l W(v)$. We can further show that $ L(D_{v_l, k})$ is a translation of $L(V_P^k)$ by $d_l \ell^k$ for each $l$ and $k$. Using the expression (\ref{eq:p_v}) we compare $L(v)$ and $L(u)$ for a corresponding pair under the above bijection. If $v\in V_P^k$ with $W(v)=l_1 \cdots l_k$, the $L(v) =(i,p_v)$ satisfies
\begin{equation*}
 p_v= d_{l_1} \ell^{k-1} + \cdots + d_{l_{k-1}} \ell + p_{l_{k}}.
\end{equation*} The corresponding vertex $u\in D_{v_l, k}$ with $W(u)= l \,l_1 \cdots l_k$ has label $L(u) = (i,p_u)$ where
$$
p_u= d_l \ell^k + d_{l_1} \ell^{k-1}+\cdots + d_{l_{k-1}} \ell + p_{l_{k}}.
$$The we have the desired difference $p_u-p_v= d_l\ell^k$. Therefore the above bijection between $V_P^{k}$ and $D_{v_l, k}$ transforms into a translation of $L(V_P^k)$ by $d_l \ell^k$ for all $l$ and $k$, i.e., the following diagram commute
$$
\CD
V_P^k@ >>>D_{v_l,k}\\
@VV{L}V@VV{L}V\\
L(V_P^k)\subset R_i@>{+d_l\ell^k}>>L(D_{v_l,k})\subset R_i
\endCD
$$

We are ready to find a polynomial $\rho(k)$ such that $|\phi^k(\tau)|\leq \rho(k)\ell^k$ for all $k$ when $P=(i,p)$ and $Q=(j,q)$ are stable points where $1\leq i,j \leq n$. Note that we want to make sure that the degree of $\rho(k)$ does not depend on $k$. Let $V_Q^1=\{u_1, \cdots, u_\ell\}$ with $L(u_l) = (j, q_l) = (Q) \tp_l$, and let $d_l'=q_l-q$ for $1\leq i\leq \ell$. To deal with the base case $k=1$, we shall take $\rho(k)$ with $p(1)=|\phi(\tau)|$. Corollary \ref{coro:prod_transposition} implies that
$$
\E(\phi^{k+1}(\tau)) = \prod_{\omega\in \Omega_{\ell, k+1}} \left(P_\omega, Q_\omega\right).
$$
To establish desired bounds for $|\phi^{k+1}(\tau)|$ let us regroup those $\ell^{k+1}$ transpositions into $\ell$ subcollections as follows. The sets $V_P^1=\{v_1, \cdots, v_\ell\}$ and $V_Q^1=\{u_1, \cdots, u_\ell\}$ provide canonical decompositions
$$
V_P^{k+1} =\bigsqcup_{1\leq l\leq \ell} D_{v_l, k}\text { and  } V_Q^{k+1} =\bigsqcup_{1\leq l\leq \ell} D_{u_l, k}.
$$
Let $\sigma_l$ be the product
$$
\sigma_l = \prod_{\omega\in W(D_{v_l,k})} \left(P_\omega, Q_\omega\right)=\prod_{\substack{\omega= l \omega'\\ \omega'\in \Omega_{\ell, k}}} \left(P_\omega, Q_\omega\right)
$$
for $1\leq l\leq \ell$. Note that each $\sigma_l$ is the restriction of $\E(\phi^{k+1}(\tau))$ on a $\phi^{k+1} (\tau)$-invariant subset $L(D_{v_l,k}) \cup L(D_{u_l,k})$. So $ \E(\phi^{k+1}(\tau)) = \sigma_1 \cdots \sigma_\ell$.
Since $L(D_{v_l,k})\subset R_i$ and $L(D_{u_l,k})\subset R_{j}$ are translations of $L(V_P^k)$ and $L(V_Q^k)$ by $d_l\ell^k$ and $d'_l\ell^k$ respectively for each $l$, $\sigma_l$ can be expressed as a conjugation of $\E(\phi^k(\tau))$. We need to consider three cases depending on $i$ and $j$. \\
\textbf{Case I.} $2\leq i\neq j \leq n$. Define $\beta_l$ by
$$
\beta_l = g_i^{d_l \ell^k} g_{j}^{d_l' \ell^k}
$$for $1\leq l\leq \ell$. Then $\sigma_l$ is the conjugation of $\phi^k(\tau)$ by $\beta_l$. Setting $d=1/4 \max\{d_1, \cdots, d_\ell, d_1', \cdots, d_\ell'\}$ we can expect
$$
|\sigma_l| = |\E(\phi^k(\tau))|+2|\beta_l| \leq \rho(k)\ell^k+2(d_l+d_l')\ell^k \leq \Bigl(\rho(k)+d \Bigr)\ell^k
$$ for each $l$. Obviously there exists a polynomial $\rho(k)$ of degree $d$ with the condition  $p(1)=|\phi(\tau)|$ such that
\begin{equation}\label{eq:bound by poly}
|\E(\phi^{k+1}(\tau))| \leq \sum_{1\leq \ell } |\sigma_l| = \Bigl(\rho(k)+d \Bigr)\ell^{k+1} <\rho(k+1) \ell^{k+1}.
\end{equation} For example one can take $\rho(k)= k^d+ |\phi(\tau)|$.

\textbf{Case II.} $j=1$, $2\leq i\leq n$. Take $i'\neq i$ to define $\beta_l$ as
$$
\beta_l = (g_{i'}^{-1} g_i)^{d_l \ell^k} g_{i'}^{-d_l' \ell^k}
$$ for $1\leq l\leq \ell$. Since each $\sigma_l$ is the conjugation of $\phi^k(\tau)$ by  $\beta_l$ for each $l$, we have
$$
|\sigma_l| = |\E(\phi^k(\tau))|+2|\beta_l| \leq \rho(k)\ell^k+2(2d_l+d_l')\ell^k \leq \Bigl(\rho(k)+d \Bigr)\ell^k
$$ where $d= 1/6 \max\{d_1, \cdots, d_\ell, d_1', \cdots, d_\ell'\}$. Therefore we can choose a polynomial $\rho(k)$ which satisfies the inequality (\ref{eq:bound by poly}) as well as the condition on $p(1)$.

\textbf{Case III.} $i=j$. In this case it may not be useful to express $\sigma_l$ as a conjugation of $ \E(\phi^k(\tau))$ by $\beta_l$. It seems that $|\beta_l|$ can be larger than we expected as $\beta_l$ translates $L(D_{v_l,k})$ and $L(D_{u_l,k})$ independently by distinct amounts $d_l\ell^k$ and $d'_l \ell^k$ on the same ray. Instead of induction on $k$, we can use a pattern that the transpositions of $\E(\phi^{k}(\tau))$ follow. Observe that if $P=(i,p)$ and $Q=(i,q)$ are stable points, then $L(V_P^k)$ is a translation of $L(V_Q^k)$ on the ray $R_i$ for each $k\in \bbn$. Assume that $q<p$ and $2\leq i$. For $k=1$, it is each to check that
$$
p_l-q_l = \Bigl(m_{l,i} + \ell(p-k_{l,i}\Bigr)-\Bigl(m_{l,i} + \ell(q-k_{l,i}\Bigr)=\ell(p-q)
$$for each $l$. So $L(V_P^1)$ is the translation of $L(V_Q^1)$ by $\ell(p-q)$. Since
$$
d_l -d'_l = (p_l-p)- (q_l-q) = (p_l-q_l) -(p-q) = (\ell-1)(p-q)
$$ for all $l$, we can use the expression (\ref{eq:p_v}) to compare $L(v)=(i,p_v)$ and $L(u)=(i,q_u)$ for each pair of vertices $v\in V_P^k$ and $u\in V_Q^k$ with $W(v)=W(u)$. If $W(v)=l_1 l_2 \cdots l_k$, then we have
\begin{align*}
p_v-q_u &= \Bigl(d_{l_1} \ell^{k-1}  +\cdots + d_{l_{k-1}} \ell + p_{l_{k}}\Bigr) -\Bigl(d\,'_{l_1} \ell^{k-1} + \cdots + d\,'_{l_{k-1}} \ell + q_{l_{k}}\Bigr) \\
&=(d_{l_1}-d\,'_{l_1})\ell^{k-1}+\cdots+(d_{l_{k-1}}-d\,'_{l_{k-1}})\ell + (p_{l_{k}}-q_{l_{k}})\\
&= (\ell-1)(p-q)(\ell^{k-1}+ \cdots+\ell) + \ell(p-q),
\end{align*}which does not depend on the choice of $v$ and $u$. Therefore the whole set $L(V_P^k)$ is a translation of $L(V_Q^k)$ by $y = p_v-q_u$. In particular, all $\ell^k$ transpositions of $\phi^{k}(\tau)$ are translations of one transposition on $ R_i$. So we can write $\phi^{k}(\tau)$ as
\begin{equation}\label{eq:prod of trans d}
\E(\phi^{k}(\tau))= \prod_{1\le m\le \ell^k} (x_m, x_m+y)
\end{equation} where $(x, x')$ stands for the transposition exchanging $(i, x)$ and $(i, x')$. Since all transpositions in (\ref{eq:prod of trans d}) commute with each other we can arrange them so that $x_m< x_{m'}$ for all $m<m'$.
Set $\tau_1=(x_1, x_1+y)$ to express $(x_m, x_m+y)$ as a conjugation of $\tau_1$ by ${g_i}^{x_m-x_1}$ for all $m$. Cancelling subwords $g_i g_i^{-1}$ we obtain
\begin{align*}
|\E(\phi^{k}(\tau)) |&= |\tau_1 \tau_1^{{g_i}^{x_2-x_1}} \cdots \tau_1^{{g_i}^{x_m-x_1}}|\\
&=|\tau_1 {g_i}^{-(x_2-x_1)}\tau_1 {g_i}^{x_2-x_1} {g_i}^{-(x_3-x_1)} \cdots {g_i}^{-(x_m-x_1)}\tau_1 {g_i}^{x_m-x_1}|\\
&\leq \ell|\tau_1| + 2(x_m-x_1)
\end{align*}
To bound $|\tau_1|$ we can use the fact that a transposition $\tau_0$ with $\supp \tau_0 \subset  B_{n,r}$ has length at most $5r$. Proposition~\ref{prop:stable points} implies that $\supp \tau_1 \subset B_{n,r}$ with $r\le A_0\ell^k +s$ for some constant $A_0$ which does not depend on $k$. Thus $|\tau_1|\leq 5( A_0\ell^k +s)$. Proposition~\ref{prop:stable points} also implies that
$x_m+y\leq A_0\ell^k +s$ for all $m$. So
$$
|\E(\phi^{k}(\tau)) |\leq 5\ell( A_0\ell^k +s) + 2 ( A_0\ell^k +s)
$$ for all $k$. Therefore there exists a polynomial $\rho(k)=5A_0\ell + 2 A_0+5s\ell +2s$ such that  $|\E(\phi^{k}(\tau) )|\leq \rho(k)\ell^k$ for all $k\in \bbn$.

In case $P$ and $Q$ are stable points of $R_1$ one can apply similar argument and calculation after expressing transpositions of (\ref{eq:prod of trans d}) as conjugations of $\tau_1$ by negative powers of ${g_i}$. So far we have shown that if $\tau$ exchanges two stable points there exists a polynomial $\rho$ such that
\begin{equation}\label{eq:bound by poly}
 |\E(\phi^{k}(\tau)) |\leq \rho(k)\ell^k
\end{equation}
for all $k$ where the degree of $\rho(k)$ does not depend on $k$.

Suppose $P,Q\in \ES$. For each $k\in \bbn$, $V_P^k$ and $V_Q^k$ can be expressed as disjoint unions of $N_P(k)$ and $N_Q(k)$ intervals respectively. The two collections of intervals determine two partitions on $\{1, \cdots, \ell^k\}$ via the map $A: V_P^k \to \{1, \cdots, \ell^k\}$. The common refinement of two partitions determine a set of common intervals for both $V_P^k$ and $V_Q^k$. Let $\calI_k=\{ I_1, \cdots, I_N\}$ and $\calI'_k=\{ I'_1, \cdots, I'_N\}$ denote the common intervals for $V_P^k$ and $V_Q^k$ respectively, i.e., $W(I_j)=W(I_j')$ for all $j$. Lemma~\ref{lm:bound on interval} guarantees that $N\leq 2A_1 k\ell$ where $A_1$ is a constant determined by $\phi$. Using those common intervals we can write $
\E(\phi(\tau)^k) = \prod_{1\le j\le N} \sigma_{j}$ where
$$
\sigma_{j} = \prod_{\omega\in W(I_j)} \left(P_\omega, Q_\omega\right)
$$
for all $k$.
Note that an interval of has size $\ell^m$ for some $m\in \bbn$ by definition. If an interval $I_j$ has size $\ell$ then $\sigma_{j}$ is a product of $\ell$ transpositions on $L(I_j) \cup L(I'_j)$ whose supports belong to a ball of radius $A_0\ell^k +s$. So we have a bound
\begin{equation}\label{eq:poly bound on sigma j 1}
|\sigma_{j}|\leq 5(A_0\ell^k+s)\leq (5A_0+s)\ell^k =\rho_j(k) \ell^k.
\end{equation} If an interval $I_j$ has size $\ell^{m}$ with $2\leq m\leq k$ then $I_j$ and $I'_j$ were common subintervals of original intervals for $V_P^k$ and $V_Q^k$ that we started with. Thus there exist $u\in V_P$ and $v\in V_Q$ such that $I_j=D_{u,m}\cap V_P^k$, and $I'_j=D_{v,m}\cap V_Q^k$. Since both $L(u)$ and $L(v)$ are stable points, we have a polynomial as in (\ref{eq:bound by poly}) to bound $|\sigma_j|$ by
\begin{equation}\label{eq:poly bound on sigma j 2}
 |\sigma_{j}|\leq \rho_j(m)\ell^{m}.
\end{equation}

So far we have found upper bounds for individual $\sigma_{j}$ depending on the size of the interval $I_j$. To obtain a universal bound consider the set of all polynomials $\{\rho_1, \cdots, \rho_N\}$ that we need for (\ref{eq:poly bound on sigma j 1}) and (\ref{eq:poly bound on sigma j 2}). Let $\rho$ denote the polynomial in the above set with the largest degree (or fastest growth).  Assuming that $\rho$ is increasing on $\bbn$, we have
\begin{equation}\label{eq:poly bound on sigma j}
 |\sigma_{j}|\leq  \rho(k)\ell^{k}
\end{equation}for all $j$. Thus we have found a polynomial $\ell \rho(k)$ such that
$$
|\E(\phi(\tau)^k)| = \sum_{1\le j\le N} |\sigma_j| \leq \ell\rho(k) \ell^{k}
$$ for all $k\in \bbn$.
\end{proof}

We remark that the expression (\ref{eq:p_v}) can be used to show that $V_P^k\subset B_{n,r}$ with $r$ bounded by a linear term of $\ell^k$ as in Proposition \ref{prop:stable points}.

\begin{Thm}\label{thm:GR bound}
Let $\phi$ be a monomorphism of $\calH_n$ with $ 2\leq n$. Then
$\GR(\phi)= \ell$ where $\ell =\ell(\phi)$.
\end{Thm}
\begin{proof}
If $\ell(\phi)=1$ it suffices to establish an upper bound since $\GR(\phi)<1$ means that $\phi$ is an eventually trivial map. In view of Proposition \ref{cor:key lemma} let us assume
$\phi(g_i) = g_if_i$ for each $i$ with $f_i \in \FSym_n$. Each $f_i$ is a product of transpositions; $f_i = \tau_1, \cdots, \tau_{F_i}$. Lemma~\ref{lm:bound for ell=1} implies that there exists a constant $A_2$ such that
$$
|\phi^k(f_i)|=| \phi^k(\tau_1) \cdots \phi^k(\tau_{F_i})|\leq  A_2F_i k
$$ for all $2\le i\le n$ and $k\in \bbn$. Let $F$ denote the maximum of $F_i$'s. From $\phi^{k+1} (g_i) = \phi^k(g_i) \phi^k(f_i)$, we have
$$
|\phi^{k+1} (g_i)| \leq | \phi^k(g_i) |+|\phi^k(f_i)| \leq | \phi^k(g_i) |+A_2 F k
$$for all $i$ and $k$. Take a polynomial $\rho$ such that
$$
\rho(k) + A_2 Fk \leq \rho(k+1)\text {  and  } G\leq\rho(1)
$$ where $G=\max\{ |\phi(g_i)|: 2\leq i\le n\}$. By induction on $k$, we check that
$$
|\phi^{k+1} (g_i)| \leq  \rho(k) +A_2 F k\leq \rho(k+1)
$$for all $i$. Thus we have a desired upper bound
$$
\GR(\phi) \le\lim_{k\to\infty}\rho(k)^{1/k}=1.
$$Therefore every monomorphism $\phi$ with $\ell(\phi)=1$ satisfies $\GR(\phi) \leq 1$ since $\GR(\phi)^d=\GR(\phi^d) \leq 1$ by Proposition \ref{cor:key lemma}.

Let us assume that $\ell(\phi)\geq 2$ and that $\phi$ satisfies $\phi(g_i) = g_i^\ell f_i
$ for all $2\le i\le n$ with $f_i\in \FSym_n$. We first establish an upper bound for $\GR(\phi)$. Each $f_i$ can be written as a product of $F_i$ transpositions, $f_i = \tau_1, \cdots, \tau_{F_i}$. By Lemma~\ref{lm:bound on transpositions} we can find polynomials $\rho_1, \cdots , \rho_{F_i}$ such that
$$
|\phi^k(f_i)|=| \phi^k(\tau_1) \cdots \phi^k(\tau_{F_i})|\leq \Bigl(\rho_1(k)+\cdots +\rho_{F_i}(k)\Bigr) \ell^k
$$for all $2\leq i\leq n$ and $k\in \bbn$. Taking $\rho$ to be the polynomial with largest degree over all $F_2 +\cdots +F_n$ polynomials, we have
$$
|\phi^k(f_i)| \leq F\rho(k)\ell^{k}
$$ for all $i$ and $k$ where $F$ denotes the maximum over all $F_i$'s. Since $ \phi^{k+1} (g_i) = \bigl(\phi^k(g_i)\bigr)^\ell \phi^k(f_i)$, we have
\begin{equation}\label{eq:half bound}
|\phi^{k+1} (g_i)| \leq\ell | \phi^k(g_i) |+|\phi^k(f_i)| \leq \ell | \phi^k(g_i) |+F\rho(k)\ell^{k}
\end{equation}for all $k$. We seek a polynomial $\bar{\rho}_i$ such that $|\phi^k(g_i)|\leq \bar{\rho}_i(k) \ell^k$ for all $k$ and $i$. There is no obstruction to take a polynomial $\bar{\rho_i}$ such that
$$\bar{\rho_i}(k) + F\rho(k) \leq \bar{\rho_i}(k+1) \text{  and  }|\phi(g_i)|\leq \bar{\rho_i}(1)
$$ for each $i$. Take a polynomial $\bar{\rho}$ such that $\bar{\rho_i} (k)\leq \rho(k)$ for all $i$ and $k$. Now (\ref{eq:half bound}) becomes
$$
|\phi^{k+1} (g_i)| \leq \bar{\rho}(k) \ell^{k+1}+ F\rho(k)\ell^{k} \leq \Bigl(\bar{\rho}(k) + F\rho(k)\Bigr)\ell^{k+1} \leq \bar{\rho}(k+1)\ell^{k+1}
$$for all $i$. Therefore we have a universal bound $|\phi^{k} (g_i)| \leq\bar{\rho}(k)\ell^{k} $ for all $i$ and $k$ where ${\bar{\rho}}$ is a polynomial on $k$. We have a desired upper bound
$$
\GR(\phi) \le\lim_{k\to\infty}\Bigl({\bar{\rho}(k)\ell^{k}}\Bigr)^{1/k}=\ell
$$
The induced homomorphism $\bar{\phi}$ on the abelianization of $\calH_n$ maps $g_i$ to $g_i^\ell$ for all $i$. So $\GR(\phi)\geq\GR(\bar{\phi})=\ell$. So far we have shown that $\GR(\phi)=\ell (\phi)$ if $\phi$ satisfies the above assumption. Now Proposition \ref{cor:key lemma} completes the proof because $\GR(\phi)^d=\GR(\phi^d) =\ell^d$.
\end{proof}

Recall that $R_1$ is the common source of all generators $g_2, \cdots, g_{n-1}$ and that each of them has a unique target such that targets are pairwise distinct. Lemma \ref{lemma:partial_translation} tells us that the behavior of $\phi(g_i)$'s is similar to this up to an element of $\Sigma_n$, the group of outer automorphisms of $\calH_n$ described in Theorem \ref{auto}.
Observe that a monomorphism $\phi$ of $\calH_n$ defines a permutation $\delta_\phi$ on $\{1, \cdots, n\}$. For each $i$, $\phi(g_i)$ has a unique target, which defines the bijection $\gamma$ as in (\ref{eq:permutation on n-1}); the target of $\phi(g_i)$ is $R_{\gamma(i)}$. We extend $\gamma:\{2, \cdots, n\} \hookrightarrow \{1, \cdots, n\}$ to get a bijection $\delta_\phi$ on $\{1, \cdots, n\}$ by setting $R_{\delta_\phi(1)}$ to be the unique source of $\phi(g_i)$'s. Note that $\delta_\phi$ determines an arrangement of $n$ rays
\begin{equation}\label{eq:delta}
R_1 \mapsto R_{\delta_\phi(1)} \; \text {and   }R_i\mapsto R_{\gamma(i)}=R_{\delta_\phi(i)}
\end{equation}
for $2\leq i\leq n$.
Lemma \ref{lemma:partial_translation} guarantees that $\delta_\phi$ is a bijection.

Being an element of the symmetric group $\Sigma_n$, $\delta=\delta_\phi$ determines a permutation matrix of $A_\delta\in \mathrm{GL}(n,\bbz)$, which has exactly one entry of $1$ in each row and each column and $0$'s elsewhere, with respect to the standard basis $\{e_1, \cdots, e_n\}$ of $\bbz^n$. The matrix $A_\delta$ has the form
\begin{align*}
\bordermatrix{
&& i&\cr
&&\vdots&\cr
\delta(i)&\cdots&1&\cdots\cr
&&\vdots&\cr
}.
\end{align*} Under the correspondence $R_i \leftrightarrow e_i$ $A_\delta$ realizes $\delta$ in (\ref{eq:delta}) as $e_{\delta(i)} = A_\delta (e_i)$ for each $i$. Recall that $\pi:\calH_n \to \bbz^n$ measures translation lengths on the rays; $\pi(g_i) = e_i-e_1$ for all $i$. The image $\pi(\calH_n)=\{(m_1,\cdots, m_n) \in \bbz^n| \sum m_i=0\}$ is freely generated by $\bar{g}_2, \cdots, \bar{g}_n$ where $\bar{g}= \pi(g)$.
The map $A_\delta$ restricts on the $A_\delta$-invariant subgroup $\pi(\calH_n)$ to define a matrix $\overbar{A}_\delta \in \mathrm{GL}(n-1,\bbz)$ with respect to the ordered basis $\{\bar{g}_2, \cdots, \bar{g}_n\}$.\\
\indent On the other hand, the induced map $\overbar{\phi}$ in the following commutative diagram can be expressed as a matrix $A_{\overbar{\phi\,}}$ with respect to the ordered basis $\{\bar{g}_2, \cdots, \bar{g}_n\}$ of $\bbz^{n-1}$.
$$
\CD
1@>>>[\calH_n,\calH_n]@>>>\calH_n@>{\pi}>>\bbz^{n-1}@>>>1\\
@.@VV{\phi'}V@VV{\phi}V@VV{\overbar{\phi\,}}V\\
1@>>>[\calH_n,\calH_n]@>>>\calH_n@>{\pi}>>\bbz^{n-1}@>>>1
\endCD
$$ 
Since $\phi(g_i)$ has a unique target $R_{\delta(i)}$ and a unique source $R_\delta(1)$, where $\phi(g_i)$ acts as a translation by $\ell$ and $-\ell$ respectively, we have
\begin{equation*}\label{eq:matrix expression}
\overbar{\phi\,}(\bar{g}_i) = \overbar{\phi(g_i) }= \ell(e_{\delta(i)} -e_{\delta(1)})=\ell A_\phi(e_{i} -e_{1}) = \ell \overbar{A}_\delta(\bar{g}_i)
\end{equation*} for each $i$ where $\ell =\ell(\phi)$.
Thus we have
$$
A_{\overbar{\phi\,}} = \ell \overbar{A}_\phi.
$$ We also have
$$
A_{\overbar{\phi^k}}=A_{(\overbar{\phi\,})^k}=\bigl(A_{\overbar{\phi\,}}\bigr)^k = \ell^k \bigl(\overbar{A}_\phi\bigr)^k
$$ for all $k \in \bbn$. In particular, $A_{\overbar{\phi^d}} = \bigl(\overbar{A}_\phi\bigr)^d= \ell^k I$ where $d=d(\delta_\phi)$ denotes the order of $\delta_\phi$ and $I\in \mathrm{GL}(n-1,\bbz)$ is the identity matrix. This proves Proposition \ref{cor:key lemma}.

\begin{Prop}\label{cor:key lemma}
A monomorphism $\phi$ of $\calH_n$ satisfies that for each $i$, $2\leq i\leq n$,
$$\phi^d(g_i) \sim \left(g_{i}\right)^{\ell^d}
$$ where $d=d(\delta_\phi)$ and $\ell = \ell(\phi)$.
\end{Prop}

\end{document}

\bibliographystyle{siam}
\bibliography{growth_H_n_ref}
\end{document}